\documentclass{siamltex}		
\pdfoutput=1
\usepackage{graphicx,multirow,comment}
\usepackage{amssymb,amsmath}
\def\myrefs#1#2{ 
{\bigskip \noindent
{\Large \bf #2}  
 \list {[\arabic{enumi}]}{\settowidth\labelwidth{[#1]}
 \leftmargin\labelwidth 
 \advance\leftmargin\labelsep
 \usecounter{enumi} }  
 \def\newblock{\hskip .11em plus .33em minus .07em}
 \sloppy\clubpenalty4000\widowpenalty4000
 \sfcode`\.=1000\relax}  }

\def\inv{^{-1}}%
\def\nref#1{(\ref{#1})}

\def\comb#1,#2,{ \left( {#1 \atop #2 } \right)  }%
\def\prodd#1,#2,#3,{ \prod_{\scriptstyle #1 \atop\scriptstyle #2 }^{#3} }%
\def\summ#1,#2,#3,{ \sum_{\scriptstyle #1 \atop\scriptstyle #2 }^{#3} }%

\def\RR{\mathbb{R}}

\newtheorem{algor}{{\sc Algorithm}}[section]

\newtheorem{tabl}{Table}[section]

\def\betab{\begin{tabbing} 
xxxx\=xxxx\=xxx\=xx\=xx\=xx\=xx\=xx\=xx\=xx\=xx\=xx\=xx\= \kill} 
\def\entab{\end{tabbing}\vspace{-0.12in}}

\newcommand{\eq}[1]{\begin{equation}\label{#1}}
\newcommand{\en}{\end{equation}}

\newcommand{\beeq}[1]{\begin{equation}\label{#1}}
\newcommand{\eneq}{\end{equation}}

\usepackage{graphicx,amsmath,amsfonts,amssymb,subfigure}
\usepackage{algorithm}
\usepackage{algorithmic}
\usepackage[normalem]{ulem}

\def\HH{\mathcal{H}}

\def\RR{\mathbb{R}}

\def\tF{\texttt{F}}
\def\inv{^{-1}}%
\def\nref#1{(\ref{#1})}

\providecommand{\norm}[1]{\lVert#1\rVert}
\graphicspath{{./EPS/}}

\newcommand{\xpmatrix}[1]{\begin{pmatrix} #1 \end{pmatrix}}

\title{Low-rank correction methods for algebraic 
domain decomposition preconditioners
\thanks{This work was supported  by NSF under grant NSF/DMS-1216366.}}

\author{Ruipeng Li 
\thanks{Address: Computer Science \& Engineering,
        University of Minnesota, Twin Cities. 
{\tt \{rli,saad\} @cs.umn.edu}}
\and Yousef Saad\footnotemark[2]} 

\begin{document} 

\maketitle 

\begin{abstract}
This paper presents a parallel preconditioning method for distributed sparse
linear systems, based on an approximate inverse of the original matrix, that
adopts a general framework of distributed sparse matrices and exploits the
domain decomposition method and low-rank corrections.
The domain decomposition approach decouples the matrix and once inverted,
a low-rank approximation is applied by exploiting the Sherman-Morrison-Woodbury
formula,
which yields two variants of the preconditioning methods. The low-rank expansion
is computed by the Lanczos procedure with reorthogonalizations.
Numerical experiments indicate that, when combined with
Krylov subspace accelerators, this preconditioner can be efficient and robust
for solving symmetric sparse linear systems. Comparisons with other
distributed-memory preconditioning methods are presented.

\end{abstract}

\begin{keywords} 
Sherman-Morrison-Woodbury formula, low-rank approximation, distributed sparse
linear systems, parallel
  preconditioner, incomplete LU factorization, Krylov subspace method,
  domain decomposition
\end{keywords}

\section{Introduction} 

Preconditioning distributed sparse linear systems remains a challenging
problem  in   high-performance  multi-processor  environments.  Simple
domain decomposition (DD) algorithms such  as the additive Schwarz method
\cite{Dryja91additiveschwarz,ASMbook,RAS,NLA:NLA80,Smith96}  are widely  used
and
they usually  yield good  parallelism.  A  well-known problem
with these preconditioners  is that they often require  a large number
of iterations when  the number of domains used is  large. As a result,
the  benefits of  increased  parallelism is often  outweighed by  the
increased  number of  iterations.  Algebraic  MultiGrid  (AMG) methods
have  achieved a  good success  and can  be extremely  fast  when they
work. However,  their success is still somewhat  restricted to certain
types  of problems.   Methods  based on  the  Schur  complement
technique such as  the parallel Algebraic  Recursive Multilevel  Solver 
(pARMS)~\cite{NLAA:pARMS}, 
which consist  of  eliminating interior  unknowns
first and then  focus on solving in some  ways the interface unknowns,
in the reduced system,
are designed to be general-purpose. The difficulty in this type of
methods is to find  effective  and  efficient  preconditioners for  the
distributed global reduced system.
In the  approach  proposed in  the
present  work, we  do not  try to  solve the  global  Schur complement
system exactly or even  form it. Instead, we   exploit the 
Sherman-Morrison-Woodbury (SMW) formula and a low-rank 
property to define  an   approximate    inverse   type   preconditioner.

Low-rank approximations have recently gained popularity as a means to
compute preconditioners. For instance,  LU factorizations or inverse matrices
using  the $\HH$-matrix format or the closely related Hierarchically Semi-Separable (HSS) matrix format rely on representing certain
off-diagonal blocks by low-rank matrices
\cite{EngquistYing,LeBorne:2003,LeBorneGras06,XiaChanGuLu10,XiaGu10}.
The main idea of this work is inspired by the recursive Multilevel Low-Rank
(MLR) preconditioner \cite{MLR-1} targeted at SIMD-type parallel machines
such as  those equipped with Graphic Processing Units (GPUs), where traditional ILU-type
preconditioners have difficulty reaching good performance \cite{RliSaadGPU}.
Here, we adapt and extend this idea to the framework of distributed sparse
matrices via  DD methods.
We refer to a preconditioner obtained by this approach
as a DD based Low-Rank (DDLR) preconditioner.
This paper considers only symmetric matrices. Extensions to the nonsymmetric
case are possible and will be explored in our future work.
The paper is organized as follows: 
In Section~\ref{sec:backg}, we briefly introduce the
distributed sparse linear systems and discuss
the domain decomposition framework. Section~\ref{sec:DD}  presents the two
proposed strategies
for using low-rank approximations in the SMW  formula.
Parallel implementation details are presented in Section~\ref{sec:ParImp}. Numerical
results of model problems and general symmetric linear systems are presented in
Section ~\ref{sec:exp}, and we conclude in Section~\ref{sec:concl}.

\section{Background: distributed sparse linear systems}\label{sec:backg}
The parallel solution of a linear systems of the form
\begin{equation}
Ax = b, \label{eq:SYST0}
\end{equation}  
where $A$ is an $n \times n$ large sparse \emph{symmetric} matrix, 
typically begins by subdividing the problem into 
$p$ parts with the help of a graph partitioner
\cite{Aykanat99,CHACO,METIS-SIAM,kolda98partitioning,SCOTCH,Simon-RSB}.
Generally, this consists of assigning  sets of equations along
with the corresponding right-hand side values 
to subdomains.
If equation number $i$ is assigned to a given subdomain,
then it is common to also assign unknown number $i$ to the same
subdomain.
Thus, each  process holds a set of
equations (rows of the linear system) and vector components 
associated with these rows. 
This viewpoint is prevalent  when taking a purely algebraic viewpoint
for solving systems of equations that arise from Partial Differential
Equations (PDEs) or general unstructured sparse matrices.

\subsection{The local systems}
In this paper we partition the problem using an \emph{edge separator} 
as is done in the pARMS method  for example.
As shown in Figure \ref{fig:locsys}, once a graph is  partitioned, three types
of unknowns appear:
(1) Interior unknowns  that are coupled  only with
local unknowns; (2) Local  interface unknowns  that are coupled  with
both external  and local  unknowns; and (3)   External
interface unknowns that belong  to  other subdomains and  are  coupled
with local interface unknowns.
\begin{figure}[ht]
\hfill
\begin{minipage}{0.3\textwidth}
  \includegraphics[width=0.95\textwidth]{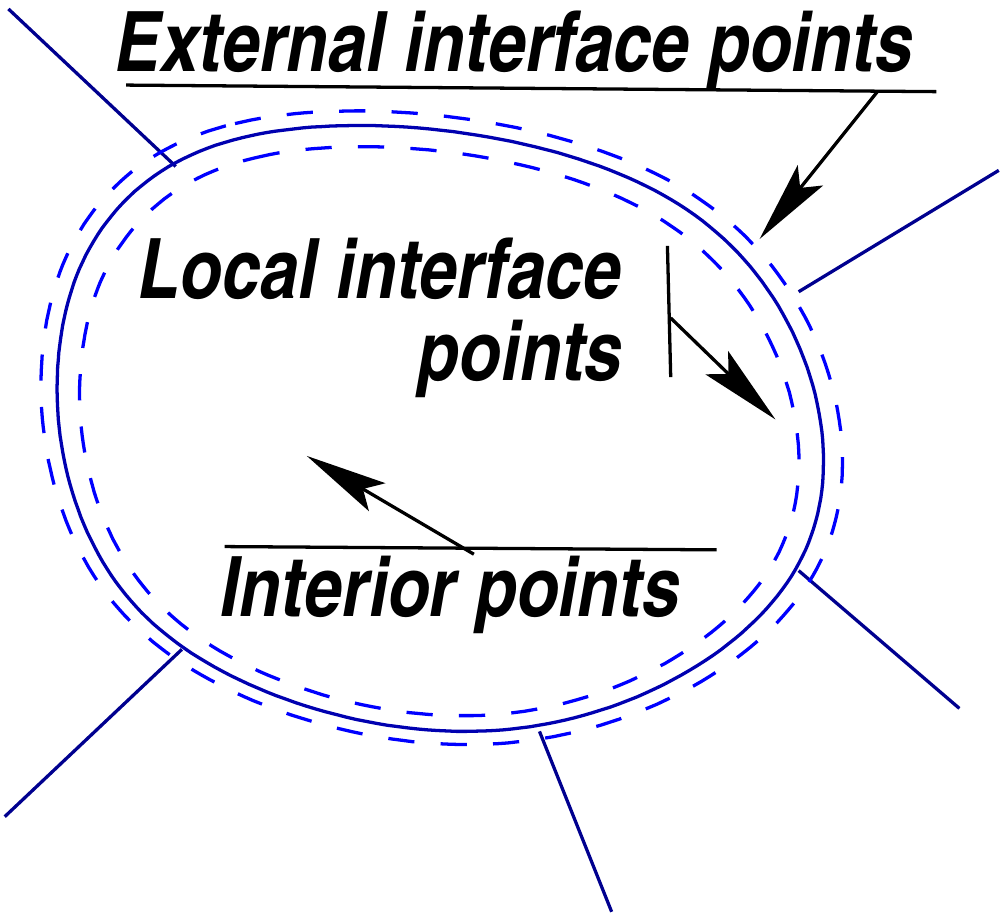} 
\end{minipage} \hfill
\begin{minipage}{0.5\textwidth}
  \hspace{0.04\textwidth}
  \includegraphics[width=0.95\textwidth]{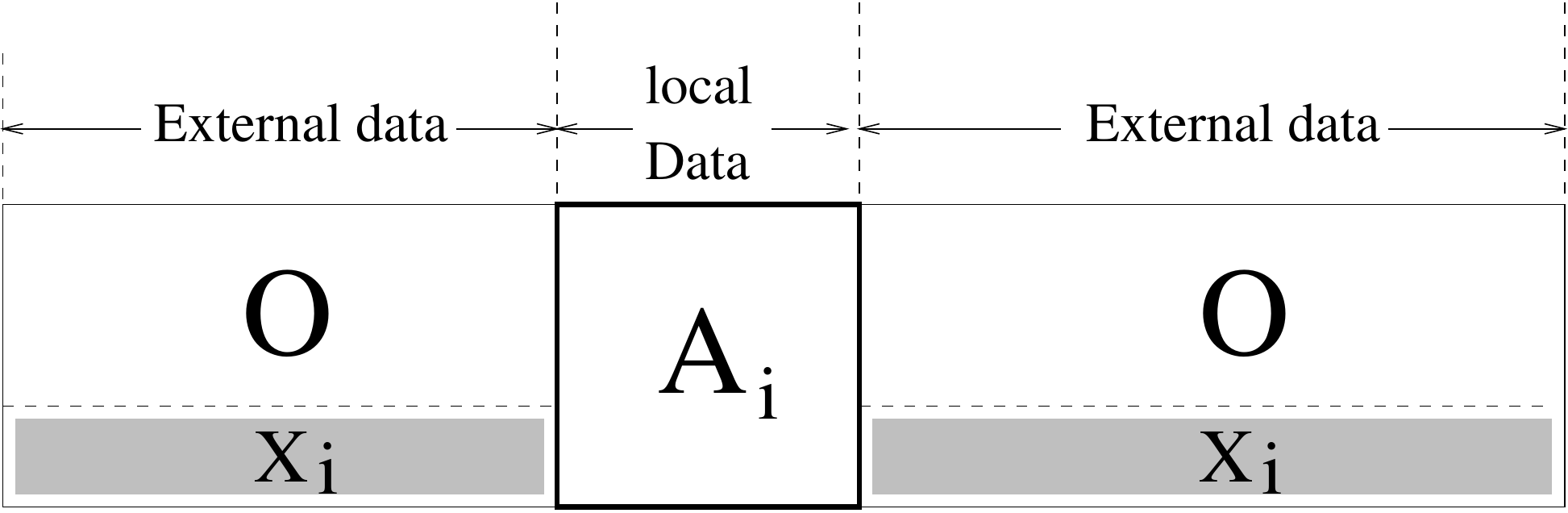} 
\end{minipage} 
\hfill \mbox{\ } 
  \caption{A local view of a distributed sparse matrix 
  (left) and its matrix representation (right).
  \label{fig:locsys}}
\end{figure}%
The rows of the matrix assigned to subdomain
$i$ can be split into two parts: a local matrix $A_i$ that acts on the local
unknowns and an interface matrix $X_i$ that acts on the external interface
unknowns. Local unknowns 
in each subdomain are reordered such that the interface 
unknowns  are listed after the interior ones.
Thus, each vector of local unknowns $x_i$ is split
into  two parts: a subvector $u_i$ of the internal  components
followed by a subvector $y_i$ of  the local interface  components.
The   right-hand-side vector $b_i$ is conformingly split into  subvectors
$f_i$ and $g_i$.
When the blocks are partitioned according to this splitting, the local 
system of equations can be written as
\begin{equation}
\underbrace{\xpmatrix{
 B_i  &   E_i \cr 
 E_i^T  &   C_i }}_{A_i} 
\underbrace{ 
 \xpmatrix{ u_i \cr y_i}}_{x_i} 
 + 
 \xpmatrix{ 0 \cr     \sum_{j \in N_i} E_{ij} y_j } 
 = 
 \underbrace{\xpmatrix{ f_i \cr g_i}}_{b_i} . \label{eq:locsys2} 
\end{equation}
Here,   $N_i$  is a set  of the indices of the
subdomains that are neighbors  to  subdomain   $i$.
The term  $E_{ij} y_j$ is a part of  the product 
which reflects the contribution  to the local equations from
the neighboring  subdomain  $j$. 
The  result of  this multiplication affects
only the local interface equations, which is indicated by the zero in 
the top part of the second  term  of   the left-hand side    of
\nref{eq:locsys2}.

\subsection{The interface and Schur complement matrices} \label{sec:Schur} 
The local system \nref{eq:locsys2} is naturally split in two
parts: the first part represented by the term
$A_i x_i$ involves only the local unknowns and the second part
contains the couplings between the local interface unknowns and
the external interface unknowns. Furthermore, the second row of the
equations in \nref{eq:locsys2}
\begin{equation}
 E_i^T u_i +  C_i y_i +  \sum_{j \in N_i} E_{ij} y_j = g_i, \label{eq:localsys}
 \end{equation}
defines both the inner-domain and the inter-domain couplings. It couples
the interior unknowns $u_i$ with the local interface unknowns $y_i$ and
the external ones, $y_j$.
An alternative way to order a global system is to group the interior unknowns
of all the subdomains together and all the interface unknowns
together as well.
The action of the operation on the left-hand side of \eqref{eq:localsys} on the
vector of all interface unknowns, i.e., the vector
$y^T = [y_1^T, y_2^T, \cdots, y_p^T]$, 
can be gathered into the following matrix $C$,
\eq{eq:Cmat} 
C \ = \ 
\xpmatrix{
C_1      & E_{12}    &   \ldots    &   E_{1p} \cr
E_{21}   & C_2       &   \ldots    & E_{2p} \cr
\vdots   &           &   \ddots    &  \vdots \cr
E_{p1}   & E_{p,2}   &   \ldots    & C_p 
}.
\en
Thus, if
we reorder the equations so that
the $u_i$'s are listed first followed by the $y_i$'s, we obtain
a global system which has the following form: 
\eq{eq:Bmat} 
\left(
\begin{array}{cccc|c}
B_1      &        &   &       &  \hat{E}_{1 } \cr
         & B_2    &   &       & \hat{E}_{2 } \cr
   &        &   \ddots &  &  \vdots \cr
         &        &   &   B_p       & \hat{E}_p \cr
         \hline
 \hat{E}_1^T    & \hat{E}_2^T   &   \ldots &  \hat{E}_p^T   &   C      \cr
\end{array}
\right)
\xpmatrix{
u_1 \cr u_2 \cr \vdots \cr u_p \cr y}  =
\xpmatrix{
f_1 \cr f_2 \cr \vdots \cr f_p \cr g},
\en
where $\hat E_i$ is expanded from $E_i$ by adding
zeros and on the right-hand side, $g^T = [g_1^T, g_2^T, \cdots, g_p^T]$. 
Writing the system in the form \eqref{eq:locsys2} is commonly adopted in practice 
when solving distributed sparse linear systems, while the form
\eqref{eq:Bmat} is more convenient for analysis. In what follows, we will assume
that the global matrix is put in the form of \eqref{eq:Bmat}.
The  form  \eqref{eq:locsys2} will return
in the discussions of Section \ref{sec:ParImp}, which deal with the parallel
implementations.

We will assume that  each subdomain $i$ has $d_i$ interior
unknowns and  $s_i$ interface unknowns, i.e., the length of
$u_i$ is $d_i$ and that of $y_i$ is $s_i$. We will denote by
$s$ the size of $y$, i.e., $s = s_1+s_2+\cdots+s_p$.
With this notation,  each $E_i$ 
is a matrix of size $d_i \times s_i$. The expanded version of this
matrix, $\hat E_i$ is of size $d_i \times s$ and
its columns outside of those corresponding to the unknowns in  
$y_i$ are zero. 
An illustration for $4$ subdomains is shown in Figure~\ref{fig:Bmat}.

\begin{figure}[h!]
\center
\subfigure {
\includegraphics[width=0.4\textwidth]{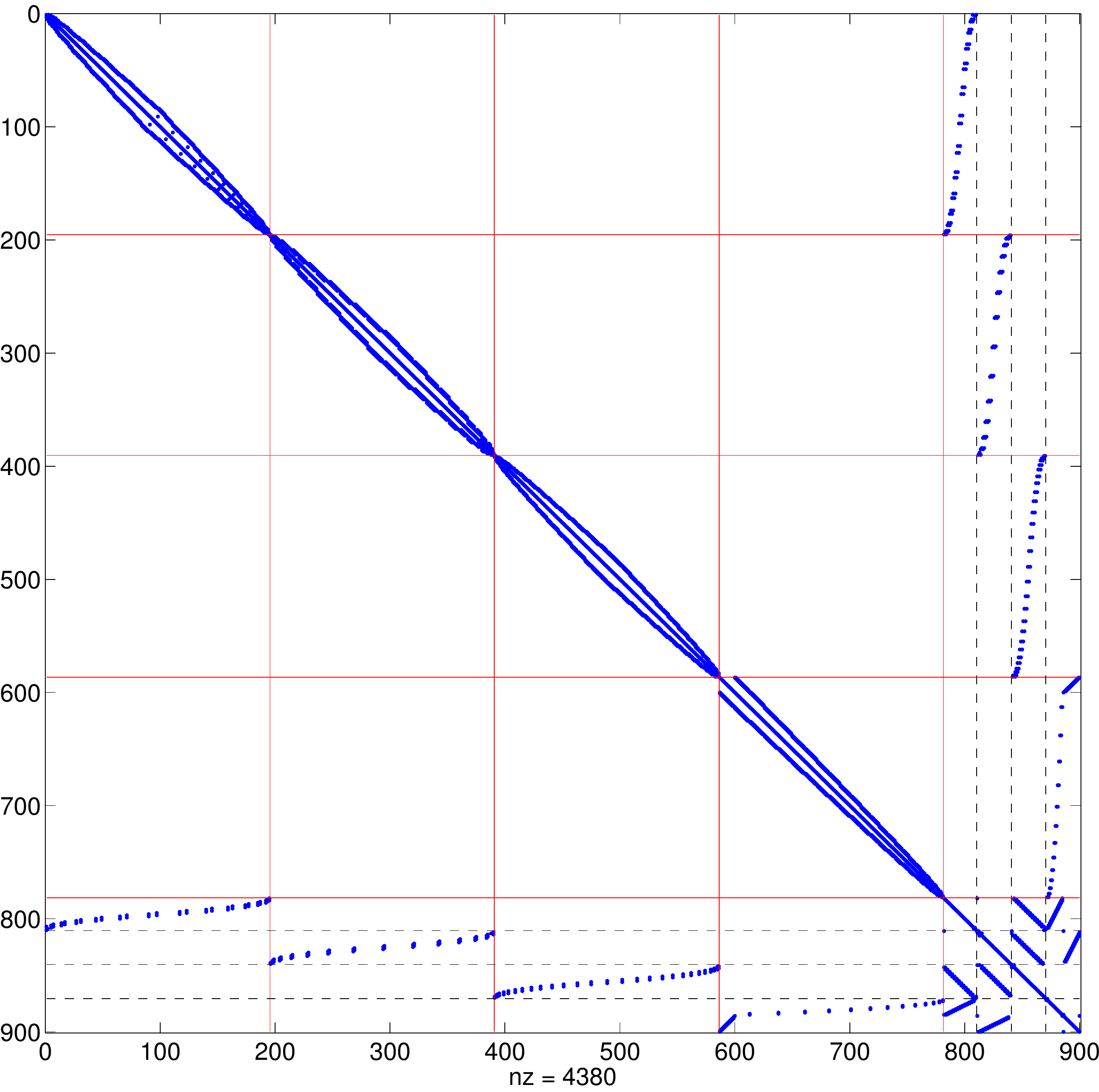}
}
\hspace*{2em}
\subfigure {
\includegraphics[width=0.4\textwidth]{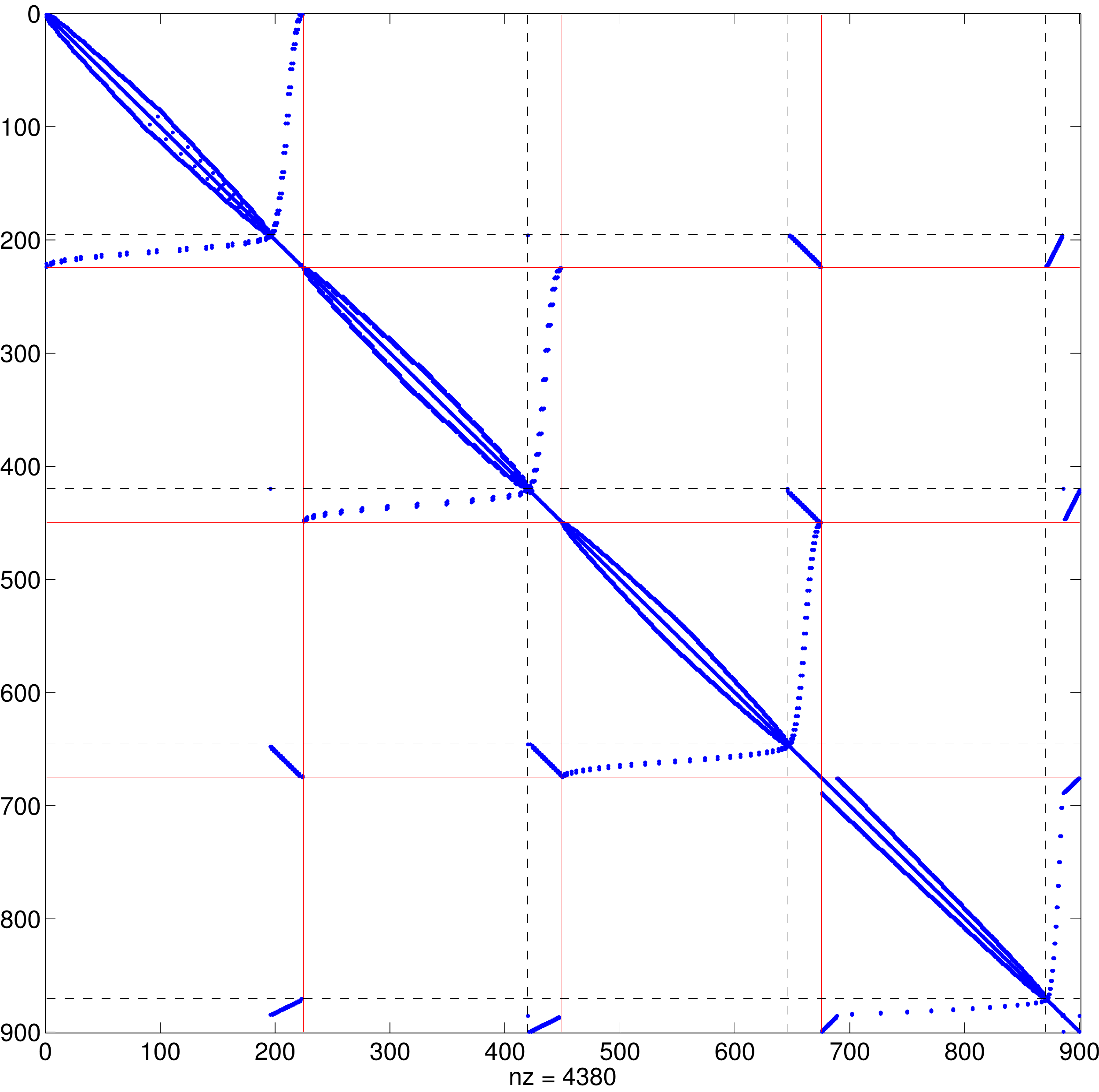}
}
\caption{An example of a 2-D Laplacian matrix which is partitioned into $4$ subdomains and reordered according to \nref{eq:Bmat} (left) and \nref{eq:locsys2} (right) respectively. \label{fig:Bmat}}
\end{figure}



A popular way of solving a global system  put into the form
of \eqref{eq:Bmat} is to exploit the Schur complement techniques
that eliminate the interior
unknowns $u_i$ first and then focus on solving in some
way for the interface unknowns. The interior unknowns can then be easily
recovered  by back substitution.  Assuming that $B_i$  is nonsingular,
$u_i$ in \eqref{eq:locsys2} can be
eliminated by means of the first equation:  $u_i=B_i\inv(f_i-E_iy_i)$ 
which yields, upon substitution in \eqref{eq:localsys},
\begin{equation}
  S_i y_i +  \sum_{j \in N_i} E_{ij} y_j = g_i - E_i^TB_i\inv f_i \equiv
g'_i, \label{eq:localschursys}
\end{equation}
in which $S_i$ is the \emph{local} Schur complement,
\begin{equation}
S_i = C_i - E_i^T B_i \inv E_i. \notag
\end{equation}
When written for each subdomain $i$, \eqref{eq:localschursys} yields the
\emph{global} Schur complement system that involves only the interface unknown
vectors $y_i$ and the reduced system  has  a natural  block  structure,
\begin{equation}
\xpmatrix{
S_1      & E_{12}    &   \ldots   &   E_{1p} \cr
E_{21}   & S_2       &  \ldots    & E_{2p} \cr
\vdots   &           &   \ddots   &  \vdots \cr
E_{p1}   & E_{p,2} & \ldots     & S_p
}
\xpmatrix{ y_1 \cr y_2 \cr \vdots \cr y_p}
=
\xpmatrix{ g'_1 \cr g'_2 \cr \vdots \cr g'_p}
.
\label{eq:SchSys}
\end{equation}
Each of the diagonal blocks in  this  system  is the local Schur complement
matrix $S_i$, which is
dense in general.  The off-diagonal blocks  $E_{ij}$
 are identical with those of the local system
\nref{eq:Cmat} and are sparse.
A  key idea    here is to
(approximately) solve the reduced
system \nref{eq:SchSys} efficiently.
For example, in pARMS~\cite{NLAA:pARMS}  efficient preconditioners
are developed based on forming an approximation to the Schur complement
system and then approximately solving \nref{eq:SchSys} and then extracting
the  internal unknowns $u_i$. This defines a  preconditioning operation for the 
global system.
In the method proposed in this paper we do not try to solve the global Schur complement 
system or even form it. Instead,  an approximate inverse preconditioner to the original matrix is
obtained by exploiting a low-rank property and the SMW  formula.

\section{Domain decomposition with local low-rank  corrections}\label{sec:DD}
The coefficient matrix of the system \eqref{eq:Bmat} is of the form
\eq{eq:Global}
A \equiv 
\begin{pmatrix}
B   & \hat E \cr \hat E^T & C \end{pmatrix},
\en
where $B \in \RR^{m \times m}$, $\hat E \in \RR^{m \times s}$ and $C
\in \RR^{s \times s}$. Here, we abuse notation by using the same symbol $A$
to represent the permuted version of the matrix in \nref{eq:SYST0}.
The goal of
this  section  is to build a preconditioner for the matrix
\nref{eq:Global}.
 
\subsection{Splitting} 
We begin by splitting  matrix $A$ as follows
\eq{eq:splitA} 
A =
\begin{pmatrix} 
B & \hat  E \cr
\hat E^T & C 
\end{pmatrix} 
= 
\begin{pmatrix} 
B &     \cr
     & C 
\end{pmatrix} 
+ \begin{pmatrix} 
  & \hat  E \cr
\hat E^T &   
\end{pmatrix}, 
\en
and defining the $n \times s$ matrix,
\eq{eq:EF}
E \equiv 
\begin{pmatrix} 
\alpha^{-1} \hat E \cr - \alpha I
\end{pmatrix},
\en
where $I$ is the $s \times s$ identity matrix and $\alpha$ is a parameter. Then
from \nref{eq:splitA} we
immediately get the identity,
\eq{eq:splitA1}
\left[\begin{array}{c|c}
B        & \hat  E \cr \hline
\hat E^T & C  \end{array} \right]
= 
\left[\begin{array}{c|c}
B + \alpha^{-2} \hat E \hat E^T  &  0    \cr \hline
  0   & C + \alpha^2 I 
\end{array} \right]  - E E^T.
\en
A remarkable property is that the operator $\hat E \hat E^T $ is 
\emph{local} in that it does not involve inter-domain couplings.
Specifically, we have the following proposition.
\begin{proposition}
Consider  
the matrix $X = \hat E \hat E^T$ and its blocks $X_{ij}$ associated
with the same blocking as for the matrix in \nref{eq:Bmat}.
Then, for $1 \le i,j \le p$ we have: 
\begin{align}
X_{ij} &=  0,  \quad \mathrm{for} \quad i \ne j \notag \\
X_{ii} &=  E_i E_i^T.   \notag
\end{align}
\end{proposition}
\begin{proof}
This follows from the fact that the columns of $\hat
E$ associated with different subdomains are structurally orthogonal illustrated
on the left side of Figure~\ref{fig:Bmat}.
\end{proof}

Thus, we can write 
\eq{eq:split} 
A = 
A_0 - E E^T, \quad 
A_0 =
\begin{pmatrix} 
B + \alpha^{-2} \hat E \hat E^T  &        \cr
    &  C + \alpha^2 I  
\end{pmatrix}  \in \mathbb{R}^{n \times n},
\en 
with the matrix $E$ defined in \nref{eq:EF}.
From \nref{eq:split} and the SMW formula,
we can derive the expression for the inverse of $A$.
First define,
\eq{eq:Gmat} 
G = I  - E^T A_0 \inv E .
\en 
Then, we have  
\eq{eq:split3} 
A\inv  =
 A_0\inv  + A_0\inv E 
(\underbrace{I  - E^T A_0 \inv E}_{G} )\inv E^T A_0\inv 
\equiv 
 A_0\inv  + A_0\inv E G\inv E^T A_0\inv . 
\en 
Note that the matrix $C$ is often \emph{strongly diagonally dominant} for
matrices arising from the discretization of PDEs,  and the
parameter $\alpha $ can serve to  improve  diagonal dominance 
in the indefinite cases.

\subsection{Low-rank approximation to the $\mathbf{G}$ matrix} \label{sec:one-sided}
In this section we will consider the case when $A$ is symmetric positive definite (SPD).
A preconditioner of the form 
\eq{eq:MappG} \notag
M\inv = A_0\inv + (A_0\inv E) \tilde G\inv (E^T A_0\inv) 
\en
can be readily obtained from \nref{eq:split3} if we had an approximation
$\tilde G\inv$ to $G\inv$. 
Note that the application of this preconditioner will involve two
solves with $A_0$ instead of only one. It will also involve a solve with
$\tilde G$ which operates  on the interface unknowns.
Let us, at least formally, assume that we know  the spectral factorization of
$E^TA_0\inv E$
\begin{equation} \notag 
H \equiv E^T A_0\inv E = U \Lambda U^T,
\end{equation}
where $H \in \RR^{s\times s}$, $U$ is unitary,
and $\Lambda$ is diagonal.
From \eqref{eq:split} we have $A_0=A+EE^T$, and thus $A_0$ is SPD
since $A$ is SPD.
Therefore, $H$ is at least symmetric positive semidefinite 
(SPSD) 
and the
following lemma
shows that its eigenvalues  are all less than one.

\begin{lemma} \label{lem:eigH}
Let
$H=E^{T}A_{0}^{-1}E$. Assume that $A$
is SPD and the matrix $I-H$ is nonsingular. Then we have $0\leq\lambda<1$,
for each eigenvalue  $\lambda$ of $H$.
\end{lemma}
\begin{proof}
From \eqref{eq:split3},
we have
\[
E^{T}A^{-1}E=H+H(I-H)^{-1}H=H\left(I+(I-H)^{-1}H\right)=H(I-H)^{-1}.
\]
Since $A$ is SPD, $E^{T}A^{-1}E$ is at least SPSD. Thus, the eigenvalues of
$H(I-H)^{-1}$ are nonnegative,
i.e., $\lambda/(1-\lambda)\geq0$. So, we have $0\leq\lambda<1$.
\end{proof}

The goal now is to see what happens if we replace $\Lambda$ by a diagonal matrix
$\tilde \Lambda$. This will include the situation when a low-rank approximation
is used for $G$ but it can also include other possibilities. Suppose that $H$ is
approximated as follows: 
\begin{equation} \label{eq:Happrox}
H \approx U \tilde \Lambda U^T.
\end{equation}
Then, from the SMW formula, the corresponding  approximation
to $G\inv$ is: 
\eq{eq:UDU}
G\inv \approx \tilde{G}\inv \equiv (I - U \tilde \Lambda U^T) \inv =
I + U [ (I-\tilde \Lambda)\inv -I] U^T . 
\en
Note in passing that the above expression can be simplified to
$U  (I-\tilde \Lambda)\inv U^T $. However, we keep the above form  because
it will still be  valid when $U$ has
only $k$ $(k<s)$ columns and $\tilde \Lambda$ is $k \times k$ diagonal, in
which case we denote by $G_k\inv$ the approximation in \eqref{eq:UDU}.
At the same time, the exact $G$ can be obtained as a special case of
\eqref{eq:UDU},
where $\tilde \Lambda$ is simply equal to $\Lambda$. Then we have
\begin{align} \label{eq:Ainv} 
A\inv  &=  A_0\inv + (A_0\inv E) G\inv (E^T A_0\inv), 
\end{align}
and the preconditioner
\begin{align}
M\inv  &=  A_0\inv + (A_0\inv E) G_k\inv  (E^T A_0\inv), 
\label{eq:Mprec} 
\end{align}
from which it follows by subtraction that
\[ 
A\inv - M\inv = 
(A_0\inv E) ( G\inv - G_k \inv ) (E^T A_0\inv), \]
and therefore,
\eq{eq:AMinv} 
A M\inv = I - A (A_0\inv E) ( G\inv - G_k \inv) (E^T A_0\inv).
\en
A first consequence of \eqref{eq:AMinv} is that there will be at lease $m$
eigenvalues
of $AM\inv $ that are equal to one, where $m=n-s$ is the dimension of $B$
in
\eqref{eq:Global} or in other words, the  number of the interior
unknowns.
From \nref{eq:UDU} we obtain
\eq{eq:diffG}
G\inv - G_k\inv = U [ (I-\Lambda)\inv - (I-\tilde \Lambda)\inv ] U^T.
\en
The simplest selection of $\tilde \Lambda$ is the one that ensures that the $k$ 
largest eigenvalues
of $(I-\tilde \Lambda)\inv$ match the largest eigenvalues
of $(I- \Lambda)\inv$. 
This simply minimizes the 2-norm  of \nref{eq:diffG} under the assumption
that the approximation in \eqref{eq:Happrox} is of rank $k$.
Assume that the eigenvalues of $H$ are $\lambda_1 \ge \lambda_2 \ge \cdots \ge
\lambda_s$. This means that the  diagonal entries $\tilde \lambda_i $  of
$\tilde \Lambda$ are selected such that
\begin{equation} \label{eq:Lambda-1}
\tilde \lambda_i = \left\{ \begin{array}{cl} 
\lambda_i &  \mbox{if} \quad  i\le k \\
 0  & \mbox{otherwise} 
\end{array}.
\right.
\end{equation}
Observe that from \nref{eq:diffG} the eigenvalues of 
$ G\inv -G_k\inv  $ are 
\[ 
\left\{ \begin{array}{cl} 
0  & \mbox{if} \quad  i\le k \\
(1-\lambda_i)^{-1} -1  & \mbox{otherwise} 
\end{array}.
\right.
\] 
Thus, from \nref{eq:AMinv} we can infer that
$k$ more eigenvalues of $A M\inv $ will take the value one 
in addition to the existing $m$ ones revealed above independently
of the choice of $\tilde G\inv$. 
Noting that $(1- \lambda_i)\inv -1 = \lambda_i/(1-\lambda_i) \ge 0$, since $0
\le \lambda_i < 1$ and 
we can say that 
the remaining $s-k$ eigenvalues of $AM\inv$ will be between $0$ and $1$.
Therefore, the result in this case is that the preconditioned matrix
$AM\inv$  in \eqref{eq:AMinv} will have $m+k$ eigenvalues  equal to one, and
$s-k$ other eigenvalues between 0 and 1.

From an implementation point of view, it is clear that a full 
diagonalization of $H$ is not needed. All we need is $U_k$,
the $s \times k$ matrix 
consisting of the first $k$ columns of $U$, along with the diagonal matrix 
$\Lambda_k$
of the corresponding eigenvalues  $\lambda_1, \cdots, \lambda_k$. Then, noting 
that \nref{eq:UDU} is still valid with $U$ replaced by $U_k$ and
$\Lambda$ replaced by $\Lambda_k$,  we can get the approximation $G_k$
and its inverse directly:
\eq{eq:Gk2}   
G_k = I - U_k \Lambda_k U_k^T,   \quad
G_k\inv =  I + U_k [ (I-\tilde \Lambda_k)\inv -I] U_k^T.
\en
It may have become clear to the reader that it is possible to select 
$\tilde \Lambda$ so that $A M\inv$ will have eigenvalues larger than one.
Consider  defining $\tilde \Lambda$ such that
\[ 
\tilde \lambda_i = \left\{ \begin{array}{cl} 
\lambda_i & \mbox{if} \quad  i\le k \\
 \theta  & \mbox{if} \quad  i > k \\
\end{array},
\right.
\]
and denote by $G_{k,\theta}\inv$ the related analogue of \nref{eq:Gk2}.
Then,  from \nref{eq:diffG} the eigenvalues of 
$ G\inv -G_{k,\theta} \inv  $ are 
\eq{eq:eigDiffG}
\left\{ \begin{array}{cl} 
0  & \mbox{if} \quad  i\le k \\
(1-\lambda_i)^{-1} -  (1-\theta)^{-1} 
 & \mbox{if} \quad  i > k \\
\end{array}.
\right.
\en
Note that for $i>k$, we have
\[
\frac{1}{1-\lambda_i} - \frac{1}{1-\theta} = 
\frac{\lambda_i -\theta}{(1-\lambda_i)(1-\theta) }, \] 
and these eigenvalues can be made negative by selecting
$ \lambda_{k+1} \le \theta < 1$ and
the choice that yields the smallest 2-norm is $\theta = \lambda_{k+1}$.
The earlier definition of $\Lambda_k$ in \eqref{eq:Lambda-1} that
truncates the eigenvalues of
$H$ to zero corresponds to selecting $\theta = 0$.

\begin{theorem}\label{thm:eigs}
Assume that $A$ is SPD and $\theta $ is selected so that
$\lambda_{k+1} \le \theta < 1 $. Then the 
eigenvalues $\eta_i $ of $AM\inv$ are such that,
\eq{eq:eigsprec} 
1 \le \eta_i \le 1+ \frac{1}{1-\theta} \ \| A^{1/2}A_0\inv E \|_2^2.
\en
Furthermore, the term $\| A^{1/2}A_0\inv E \|_2^2$ is bounded from above by a 
constant:
\begin{equation} \notag
\norm{A^{1/2}A_0\inv E}_2^2 \le \frac{1}{4}. 
\end{equation}
\end{theorem} 
\begin{proof}
We rewrite \nref{eq:AMinv} as 
$ A M\inv = I + A (A_0\inv E) ( G_k \inv - G\inv) (E^T A_0\inv)$ or
upon applying a similarity transformation with $A^{1/2}$
\eq{eq:AMinv2} 
 A^{1/2} M\inv A^{1/2}  = I + ( A^{1/2} 
  A_0\inv E) ( G_k \inv - G\inv) (E^T A_0\inv A^{1/2}).
\en 
From \nref{eq:eigDiffG} we see that for $j\le k$ 
we have $\lambda_j ( G_k \inv - G\inv ) = 0 $, and for
$j > k$,
\[
0 \le \lambda_j ( G_k \inv - G\inv ) = 
 (1-\theta)^{-1} - (1-\lambda_j)^{-1} \le  (1-\theta)^{-1}.
\]
This is because $1/(1-t)$ is an increasing function and for $j>k$, we have
$0 \le \lambda_j \le \lambda_{k+1} \le \theta$. 
The rest of the proof follows by taking the Rayleigh quotient of
an arbitrary vector $x$ and utilizing \nref{eq:AMinv2}.

For the second part, first note that 
$\norm{A^{1/2}A_0\inv E}_2^2= \rho\left(E^T
A_0\inv A
A_0\inv E\right)$, where $\rho(\cdot)$ denotes the spectral radius of a matrix.
Then, from $A=A_0-EE^T$, we have
\begin{equation} 
 E^T A_0\inv A A_0\inv E = E^TA_0\inv E - \left(E^TA_0\inv
E\right)\left(E^TA_0\inv E\right) \equiv H - H^2. \notag
\end{equation}
Lemma \ref{lem:eigH} states that each eigenvalue $\lambda$ of $H$ satisfies
$0\le \lambda<1$.
Hence, for each eigenvalue $\mu$ of $E^T A_0\inv A
A_0\inv E$, $\mu = \lambda - \lambda^2$, 
which is between 0 and 1/4 for  $\lambda \in \ [0, 1)$.
This gives the desired bound 
$\norm{A^{1/2}A_0\inv E}_2^2 \le 1/4$.
\end{proof} 

\begin{figure}[ht]
\caption{DDLR-1: eigenvalues of $AM\inv$ with $\theta=0$ (left)
and
$\theta = \lambda_{k+1}$ (right) using $k=5$ eigenvectors for a $900 \times
900$ 2-D Laplacian with $4$ subdomains and $\alpha=1$. \label{fig:SpEgs}}
\begin{center}
\includegraphics[width=0.48\textwidth]{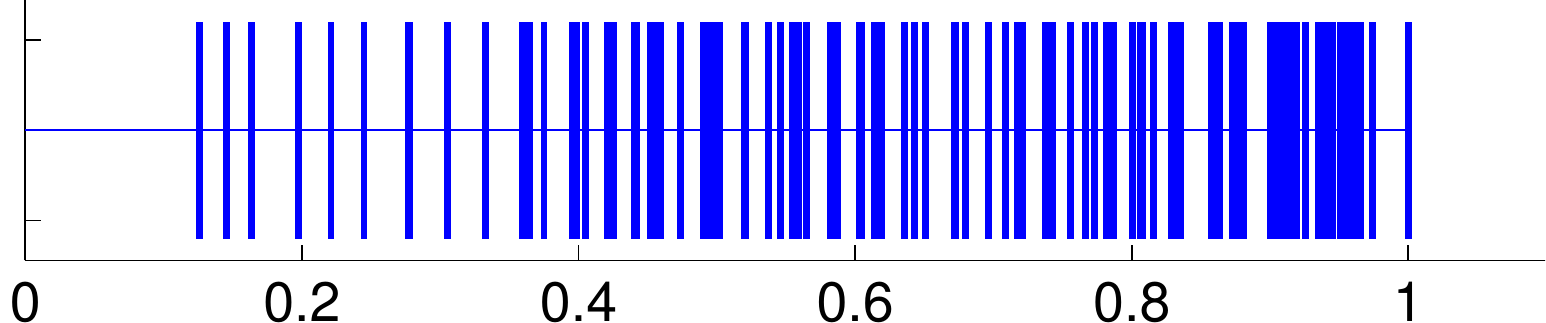} \hfill
\includegraphics[width=0.48\textwidth]{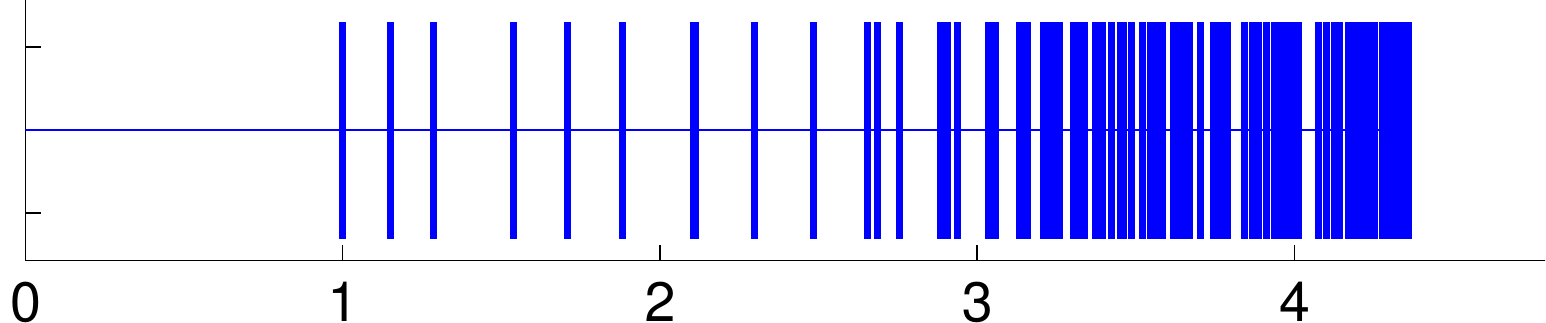}
\end{center}
\end{figure}

An illustration of the  spectra of $AM\inv$ for the two cases when
$\theta = 0$ and $\theta = \lambda_{k+1}$ with $k=5$ is shown in
Figure~\ref{fig:SpEgs}.  
The original matrix is a $900 \times 900$ 2-D Laplacian
obtained from a finite difference discretization of a square domain using $30$
mesh points in each direction. The number of the subdomains used is 4,
resulting in 119   interface unknowns.
The reordered matrix associated with this example were shown
in Figure~\ref{fig:Bmat}.

For the second choice $\theta=\lambda_{k+1}$, Theorem~\ref{thm:eigs} proved 
that 
$\|A^{1/2}A_0\inv E \|_2^2$ does not exceed $1/4$, 
regardless of the mesh size and regardless of $\alpha$, 
Numerical experiments will show that this term is close to $1/4$ for Laplacian 
matrices.
For the case with
$\alpha=1$, 
$\theta=\lambda_6\approx 0.93492$ and $\norm{A^{1/2}A_0\inv
E}_2^2\approx 0.24996$, so that the bound of the eigenvalues of $AM\inv$ given
by \nref{eq:eigsprec} is $4.8413$, which is fairly close to the largest
eigenvalue,
which is $4.3581$ (cf. the right part of Figure~\ref{fig:SpEgs}).
When $\alpha=2$,  $\theta=\lambda_6\approx 0.93987$ and
$\norm{A^{1/2}A_0\inv
E}_2^2\approx 0.25000$, so that the eigenvalue bound 
is
$5.1575$ whereas the largest eigenvalue is $4.6724$.
When $\alpha=0.5$, $\theta=\lambda_6\approx 0.96945$ and
$\norm{A^{1/2}A_0\inv
E}_2^2\approx 0.24999$, and thus the bound is $9.1840$, compared with the
largest eigenvalue $8.6917$. Therefore, we conclude for this case $\alpha=1$ 
gives the best
spectral condition number of the preconditioned matrix, which is a typical
result for SPD matrices.

We now address some implementation issues of the preconditioner related to the
second choice with $\theta = \lambda_{k+1}$. Again all that is needed are 
$U_k$, $\Lambda_k$ and $\theta$. We can show an analogue to
the expression \nref{eq:Gk2} in the following proposition.
\begin{proposition}\label{prop:gkt}
The following expression for $G_{k,\theta}\inv $ holds:
\eq{eq:gkt}
G_{k,\theta}\inv 
=
\frac{1}{1 - \theta} I  \ + \ 
U_k  \left[ (I- \Lambda_k)\inv - (1-\theta)\inv I \right] U_k^T .
\en
\end{proposition}
\begin{proof}
We write $U = [U_k, W]$, where $U_k$ is as before and $W$ contains the
remaining columns $u_{k+1}, \cdots, u_s$. Note that $W$ is not available
but we use the fact that $W W^T = I-U_k U_k^T$ for the purpose of this
proof.
With this, \nref{eq:UDU} becomes:
\begin{align*}
G_{k,\theta}\inv &= 
I + [U_k, W] \xpmatrix{  (I- \Lambda_k)\inv -I  &   \\ 
 & ((1-\theta)\inv -1) I } [U_k, W]^T   \\
&=
I + U_k \left[ (I- \Lambda_k)\inv -I \right] U_k^T 
+ \left[(1-\theta)\inv -1 \right] (I- U_k U_k^T)  \\
&= 
\frac{1}{1 - \theta} I  \ + \ 
U_k  \left[ (I- \Lambda_k)\inv - (1-\theta)\inv I \right] U_k^T .
\end{align*}
\end{proof}
\begin{proposition}\label{prop:spd-1}
Let the assumptions of Lemma \ref{lem:eigH} be satisfied.
The preconditioner  \eqref{eq:Mprec} with the matrix
$G_{k,\theta}\inv$ defined by \eqref{eq:gkt} is well-defined and
SPD when $ \theta < 1$.
\end{proposition}
\begin{proof}
From \eqref{eq:gkt}, the eigenvalues of $G_{k,\theta}\inv$ are 
$(1-\lambda_i)\inv$, $i=1,\ldots,k$ or $(1-\theta)\inv$.
Recall from Lemma \ref{lem:eigH}, $0 \le \lambda_i < 1$ for all $i$ and 
thus $G_{k,\theta}\inv$ is well-defined and SPD when $
\theta < 1$.
Hence, preconditioner \eqref{eq:Mprec} is SPD.
\end{proof}

We refer to the preconditioner \eqref{eq:Mprec} with 
$G_k\inv=G_{k,\theta}\inv$ as the \emph{one-sided} DDLR preconditioner, 
abbreviated by  DDLR-1.

\subsection{Two-sided low-rank approximation} \label{sec:two-sided}
The method to be presented in this section uses low-rank
approximations for more terms in \eqref{eq:split3}, which
yields a  preconditioner that has a simpler form. Compared with the DDLR-1
method, the resulting preconditioner is less expensive to apply and less
accurate in general.
Suppose that $A_0\inv E \in \RR^{n\times s}$ is factored in the form
\eq{eq:svd}
A_0\inv E = UV^T,
\en
as obtained from the singular value decomposition (SVD),
where $U\in \RR^{n \times s}$ and $V \in \RR^{s \times s}$ is orthogonal. Then,
for the matrix $G$ in
\eqref{eq:Gmat}, we have the following lemma.
\begin{lemma} \label{lemma:G}
 Let $G=I-E^TA_0\inv E$ as defined by \eqref{eq:Gmat} be
 nonsingular. Then,
 \begin{equation} \notag
  G\inv = I+V\left(I-U^TEV\right)\inv U^TE.
 \end{equation}
 Furthermore, the following relation holds,
 \begin{equation} \label{eq:VTGV}
  V^TG\inv V = \left(I-U^TEV\right)\inv.
 \end{equation}
\end{lemma}

\begin{proof}
For $G\inv$, we can write
\begin{align}
 G\inv = \left(I-(E^TA_0\inv) E\right)\inv = \left(I-VU^TE\right)\inv \notag 
 = I+V\left(I-U^TEV\right)\inv U^TE.
\end{align}
Relation \eqref{eq:VTGV} follows from
\begin{align}
 V^TG\inv V &= V^T(I+V\left(I-U^TEV\right)\inv U^TE)V 
=\left(I-U^TEV\right)\inv. \notag
\end{align}
\end{proof}

From \eqref{eq:svd}, the best $2$-norm
rank-$k$ approximation to $A_0\inv E$ is of the form
\eq{eq:LR}
A_0 \inv E \approx U_k V_k^T,
\en
where $U_k \in \RR^{n \times k}$ and $V_k \in \RR^{s \times k}$ with
$V_k^TV_k=I$
consist of the first $k$ columns of $U$ and $V$ respectively.
For an approximation to $G$, we define the matrix $G_k$
as
\eq{eq:Gk}
G_k = I - V_k U_k^T E \ .
\en
Then, the expression of $A\inv$ in \eqref{eq:split3} will yield the
preconditioner: 
\eq{eq:Form2} \notag
M \inv  = A_0\inv + U_k ( V_k^T  G_k\inv V_k) U_k^T.
\en
This means that we can build an approximate inverse based on a
low-rank correction of the form that avoids the use of  $V_k$ explicitly,
\eq{eq:Form2a}
M\inv = A_0\inv + U_k H_k U_k^T \quad\mbox{with} \quad
H_k = V_k^T G_k\inv V_k.
\en
Note that Lemma \ref{lemma:G} will also hold if $U$ and $V$ are replaced with
$U_k$ and $V_k$.
As a result, the matrix $H_k$ has an alternative expression that is
more amenable to computation.
Specifically, we can  show the following lemma.
\begin{lemma} \label{lem:1}
Let $G_k$ be defined by \nref{eq:Gk} and assume that  matrix
$I - U_k^T E V_k$ is nonsingular.
Then,
\eq{eq:Gkinv} \notag
G_k\inv =
I + V_k \hat H_k  U_k^T E  \quad
\mbox{with} \quad \hat H_k = (I - U_k^T E V_k)\inv .
\en
Furthermore, the following relation holds:
\eq{eq:Hk1} \notag
V_k^T G_k\inv V_k = \hat H_k
\en
i.e., the matrix $H_k$ in \nref{eq:Form2a} and the matrix
$\hat H_k$ are equal.
\end{lemma}

\begin{proof}
A proof can be directly obtained from the proof of Lemma \ref{lemma:G} by
replacing matrices $U$,$V$ and $G$ with $U_k$,$V_k$ and $G_k$ respectively.
\end{proof}

The application of \eqref{eq:Form2a} requires
one solve with $A_0$ and a low-rank correction with
$U_k$ and $H_k$. 
Since $A_0\inv E$ is
approximated on both sides of $G$ in \eqref{eq:split3},
 we refer to this preconditioner as a
\emph{two-sided} DDLR preconditioner and use the abbreviation DDLR-2.

\begin{proposition}\label{prop:spd-2}
Assume that $U_kV_k^T$ in \eqref{eq:LR} is the best $2$-norm rank-$k$
approximation to $A_0\inv E$, and that $A_0$
is SPD. Then the preconditioner given by \eqref{eq:Form2a} is well-defined 
and
SPD if and only if $\rho(U_k^T E V_k) <1$.
\end{proposition}
\begin{proof}
The proof follows from Proposition 3.2 in \cite{MLR-1} showing
the symmetry of $H_k$, and Proposition 3.4 and Theorem 
3.6 in \cite{MLR-1} for the
if-and-only-if condition.
\end{proof}

Next, we will show that the eigenvalues of the preconditioned matrix $AM\inv$ 
are between zero and one.
Suppose that $U_kV_k^T$ is obtained as in \eqref{eq:LR},
so that we have
$\left(A_0\inv E\right)V_k = U_k$.
Then, the preconditioner \eqref{eq:Form2a} can be rewritten as
\begin{align} \label{eq:Mtwosided}
M^{-1} &= A_{0}^{-1}+U_{k}H_{k}U_{k}^{T} 
=A_{0}^{-1}+\left(A_{0}^{-1}E\right)V_{k}H_{k}V_{k}^{T}\left(E^{T}A_{0}^{-1}
\right) \notag \\
&=A_{0}^{-1}+\left(A_{0}^{-1}E\right)V
\begin{pmatrix}
H_{k} & 0 \\
0 & 0
\end{pmatrix}
V^{T}\left(E^{T}A_ { 0 } ^ { -1 }
\right),
\end{align}
where $U$ and $V$ are defined in \eqref{eq:svd}.
We write $U=\left[U_k,\bar{U}\right]$ and $V=\left[V_k,\bar{V}\right]$,
where $\bar{U}$ and $\bar{V}$ consist of the $s-k$ columns
of $U$ and $V$ that are not contained in $U_k$ and $V_k$.
Recall that $H_k\inv=I-U_k^TEV_k$ and define $X=I-\bar{U}^TE\bar{V}$,
$Z=-U_k^TE\bar{V}$. 
From \eqref{eq:Ainv} and \eqref{eq:svd}, we have
\begin{equation} \label{eq:Ainv1a}
 A\inv = A_0\inv + \left(A_{0}^{-1}E\right)V\left(V^TG\inv V
\right)V^{T}\left(E^{T}A_ 0\inv\right), \notag 
\end{equation}
from which and \eqref{eq:VTGV}, it follows that
\begin{align} \label{eq:Ainv1}
A\inv &=A_0\inv + \left(A_{0}^{-1}E\right)V\left(I-U^TEV
\right)\inv V^{T}\left(E^{T}A_ { 0 } ^ { -1 }\right), \notag \\
&=A_{0}^{-1}+\left(A_{0}^{-1}E\right)V
\begin{pmatrix}
H_{k}\inv & Z \\
Z^T & X
\end{pmatrix}\inv
V^{T}\left(E^{T}A_0\inv\right).
\end{align}
Let the Schur complement of $H_k\inv$ be
\begin{equation} \label{eq:SS}
S_k=X-Z^TH_kZ \in \RR^{(s-k) \times (s-k)},
\end{equation}
and define matrix $\bar{S}_k \in \RR^{2(s-k) \times 2(s-k)}$ by
\begin{equation} \label{eq:Sbar}
\bar{S}_k= \begin{pmatrix}
  S_k\inv & -I \\
  -I & S_k
 \end{pmatrix}.
\end{equation}
Then, the following lemma shows that $S_k$ is SPD and $\bar S_k$ is SPSD.
\begin{lemma} \label{lem:S}
Assume that $G$ defined by \eqref{eq:split3} is
 nonsingular as well as the matrix
$I - U_k^T E V_k$. Then, the Schur complement $S_k$ defined by \eqref{eq:SS}
is SPD.
Moreover, matrix
$\bar S_k$
is SPSD with $s-k$ positive eigenvalues and $s-k$ zero eigenvalues.
\end{lemma}
\begin{proof}
From Lemma \ref{lem:eigH}, we can infer that the eigenvalues of $G$ are all
positive. Thus, $G$ is SPD and so is matrix $V^TG\inv V$. In the end, the
Schur complement $S_k$ is SPD when $H_k$ is nonsingular. The signs of the
eigenvalues of $\bar S_k$ can be easy revealed by a block LDL factorization.
\end{proof}

\begin{theorem}\label{thm:eigs2}
Assume that $A$ is SPD. Then the eigenvalues $\eta_i $ of $AM\inv$ with $M\inv$ 
given by \eqref{eq:Form2a} satisfy $0 < \eta_i \le 1$.
\end{theorem} 
\begin{proof}
From \eqref{eq:Mtwosided} and \eqref{eq:Ainv1}, it follows by subtraction
that
\begin{align} \notag 
A\inv - M\inv &= \left(A_{0}^{-1}E\right)V
\left[
\begin{pmatrix}
H_{k}\inv & Z \\
Z^T & X
\end{pmatrix}\inv
-
\begin{pmatrix}
H_{k} & 0 \\
0 & 0
\end{pmatrix}
\right]
 V^{T}\left(E^{T}A_0\inv\right), \notag \\
&= \left(A_{0}^{-1}E\right)V
\begin{pmatrix}
H_kZS_k\inv Z^TH_k & -H_kZS_k\inv \\
-S_k\inv Z^TH_k & S_k\inv
\end{pmatrix}
V^{T}\left(E^{T}A_0\inv\right), \notag \\
&=\left(A_{0}^{-1}E\right)V
 \begin{pmatrix}
  H_kZ & 0 \\
  0 & S_k\inv
 \end{pmatrix}
\bar S_k
  \begin{pmatrix}
  Z^TH_k & 0 \\
  0 & S_k\inv
 \end{pmatrix}
 V^{T}\left(E^{T}A_0\inv\right), \notag
\end{align}
where $\bar S_k$ is defined in \eqref{eq:Sbar}, so that
\begin{align} \notag 
AM\inv = I- A\left(A_{0}^{-1}E\right)V
 \begin{pmatrix}
  H_kZ & 0 \\
  0 & S_k\inv
 \end{pmatrix}
\bar S_k
  \begin{pmatrix}
  Z^TH_k & 0 \\
  0 & S_k\inv
 \end{pmatrix}
 V^{T}\left(E^{T}A_0\inv\right).
\end{align}
Hence, the eigenvalues of $AM\inv$, $\eta_i$, satisfy $0< \eta_i \le 1$, since
$\bar S_k$ is SPSD, and
$(n-s+k)$ of these eigenvalues are equal to one. 
\end{proof}

\begin{figure}[ht]
\caption{DDLR-2: eigenvalues of $AM\inv$ with $k=5$ eigenvectors for a $900
\times
900$ 2-D Laplacian  with $4$ subdomains and $\alpha=1$. \label{fig:SpEgs2}}
\begin{center}
\includegraphics[width=0.8\textwidth]{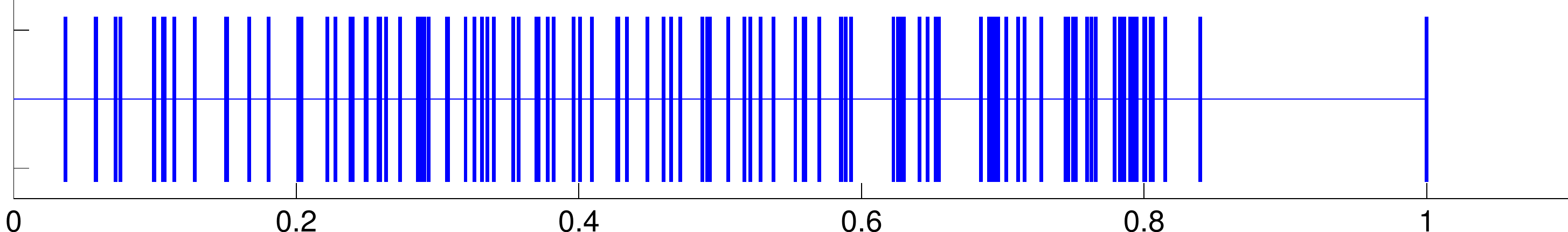}
\end{center}
\end{figure}

The spectrum of $AM\inv$ for the same matrix used for Figure~\ref{fig:SpEgs} is
shown in Figure~\ref{fig:SpEgs2}. Compared with the spectrum with the DDLR-1 method with $\theta=0$ shown in the left
part of Figure~\ref{fig:SpEgs}, the eigenvalues of the preconditioned matrix $AM\inv$ are more dispersed between
0 and 1 and the small eigenvalues are closer to zero. This suggests that the quality of the DDLR-2 preconditioner will be
lower than that of DDLR-1, which is supported by the numerical results in
Section~\ref{sec:exp}.

\section{Implementation} \label{sec:ParImp}
In this section, we  address the implementation details for building
and applying the DDLR preconditioner, especially focusing on the implementations
in a parallel/distributed environment.

\subsection{Building a DDLR preconditioner}
The construction of a DDLR preconditioner involves the following steps.
In the  first step, a graph partitioner is called on the adjacency graph to 
partition the domain. For each obtained subdomain, we separate the interior 
nodes and the interface nodes, and reorder the local matrix into
the form of \eqref{eq:locsys2}.
The second step is to build a solver for each $B_{i,\alpha} \equiv 
B_i+\alpha^{-2} E_i E_i^T$.
These two steps can be done in parallel. 
The third step is to build a solver for the \emph{global} matrix
$C_\alpha$.
We will focus on the solution methods for the linear systems with
$C_\alpha$ in Section \ref{sec:csolve}. The last step, which is the most
expensive one, is to compute the low-rank approximations. This will be
discussed  in Section \ref{sec:Lan}.

\subsection{Applying the DDLR preconditioner}
First, consider the DDLR-1 preconditioner \nref{eq:Mprec}, which we can rewrite 
as
\begin{equation}\label{eq:Mprec1}
M\inv = A_0\inv \left(I + E G_{k,\theta} \inv  E^T  A_0\inv \right).
\end{equation}
The steps involved in applying $M\inv$ to a vector $x$ are listed in Algorithm 
\ref{alg:prec1}.
The vector $u$ resulting from the last step will be the desired vector $u=M\inv
x$.
The solve with $A_0$ required in steps 1 and 5 of Algorithm \ref{alg:prec1},
can in turn be viewed
as consisting of the $p$ independent local solves with
$B_{i,\alpha}$ and the global solve with $C_\alpha\equiv C+\alpha^2  I$ as is
inferred from
\nref{eq:split}.
Recall that the matrix $C$, which has the block structure
\nref{eq:Cmat},
is the global interface matrix that couples
all the interface unknowns. So, solving
a linear system with $C_\alpha$ will require communication if 
$C_\alpha$ is assigned to different processors.
The multiplication with $E^T$ in step 2  transforms a vector of the
interior unknowns into a vector of
the interface unknowns. This can be likened to a \emph{descent} operation
that moves objects from a  ``fine'' space to a ``coarse'' space.
The multiplication with $E$ in step 4 performs the reverse operation,
which can be termed an \emph{ascent} operation, consisting of
going from the interface unknowns to the interior unknowns.
Finally, the operation with $G_{k,\theta}\inv $ in step 3 
involves all the interface unknowns, and it will also require communication.
In summary, there are essentially 4 types of operations: (1) the solve with 
$B_{i,\alpha}$; (2) the solve with $C_\alpha$;
(3) products with $E$ and $E^T$, which are dual of one another;
and (4) the application of $G_{k,\theta}\inv $ to vectors.

\begin{algorithm}[ht]
\caption{Preconditioning operations of the DDLR-1 preconditioner.
\label{alg:prec1}}
\begin{algorithmic}[1]
\STATE Solve: $A_0 z = x$  \hfill\COMMENT{$B_{i,\alpha}$ solves and $C_\alpha$
solve}
\STATE Compute: $y = E^T z$ \hfill\COMMENT{Interior unknowns to interface
neighbors}
\STATE Compute:  $w = G_{k,\theta}\inv y$  \hfill\COMMENT{Use \nref{eq:gkt}}
\STATE Compute: $v = Ew $ \hfill\COMMENT{Interface unknowns to interior
neighbors}
\STATE Solve: $A_0 u = x+v $   \hfill\COMMENT{$B_{i,\alpha}$ solves and
$C_\alpha$ solve}
\end{algorithmic}
\end{algorithm}

Next, consider the DDLR-2 preconditioner given by \eqref{eq:Form2a}.
Applying this preconditioner is much simpler, which consists of one solve with 
$A_0$ and a low-rank correction.
Communication will be required for applying the low-rank correction term,
$U_kH_kU_k^T$, to a vector because it involves all the unknowns.
We assume that the $k \times k$ matrix $H_k$ is stored on every processor.

Parallel implementations of the DDLR methods will depend on how the
interface unknowns are mapped to processors.
A few of the mapping schemes will be discussed in Section \ref{sec:pp}.

\subsection{The global solves with $\mathbf{C_\alpha}$} \label{sec:csolve}
This section addresses the solution methods for $C_\alpha$ required in both
the DDLR-1 and the DDLR-2 methods whenever solving a linear 
system with $A_0$ is needed. It is an important part of the computations,
especially for 
DDLR-1 as it takes place twice for each iteration.
In addition, it is a non-local computation and can be costly due to
the communication. 
An important characteristic of $C_\alpha$ is that it can be made strongly
diagonally dominant by selecting a proper scaling factor $\alpha$.
Therefore, the first approach one can think about is to use a few steps
of the Chebyshev iterations.
The Chebyshev method was used with a block Jacobi
preconditioner $D_\alpha$ consisting of all the local diagonal blocks
$C_i$ (see, e.g.,
\cite[\S 2.3.9]{templates} for the preconditioned Chebyshev method). 
An appealing property in the Chebyshev iterations is that no inner product is 
needed. This avoids communications among processors, which makes this method 
efficient in particular for distributed memory architectures \cite{Saad-book2}. 
The price one pays for avoiding communication is that this method
requires enough knowledge of the spectrum. 
Therefore, prior to the Chebyshev iterations, we performed a few steps of the
Lanczos iterations on the matrix pair $(C_\alpha,D_\alpha)$ \cite[\S 
9.2.6]{Saad-book1} for some 
estimates (not bounds) of the smallest and the largest eigenvalues.  
The safeguard terms used in \cite{zhousaad} were included in order to have
bounds of the spectrum (see \cite[\S 13.2]{parlettbook} for the definitions of 
these terms).

Another approach is to resort to an approximate 
inverse $X \approx C_\alpha\inv$, so that the solve with $C_\alpha$ will be 
reduced to a matrix vector product with $X$.
A simple scheme known as
the method of Hotelling and Bodewig \cite{Householder} is given by the iteration
$$
X_{k+1} = X_k (2 I - C_\alpha X_k).
$$
In the absence of dropping, this scheme squares the residual norm $\norm{I -
C_\alpha X_k}$ 
from one step to the next, so that it converges quadratically provided that
the initial guess $X_0$ is such that $\| I -C_\alpha X_0\|<1$ for some matrix
norm.
The global self-preconditioned minimal residual (MR) iterations were shown 
to have superior performance \cite{Chow-Saad-apinv}. 
We adopted this method to build an approximate inverse of $C_\alpha$. Given an
initial guess $X_0$, the self-preconditioned MR iterations can be obtained by
the
sequence of operations shown in Algorithm \ref{alg:mr}.
$X_0$ was selected as the inverse of the diagonal of
$C_\alpha$. The numerical dropping was performed by a dual threshold
strategy based on a drop tolerance and a maximum
number of nonzeros per column.

\begin{algorithm}[ht]
\caption{Self-preconditioned global MR iterations with dropping. \label{alg:mr}}
\begin{algorithmic}[1]
\STATE Compute: $R_k = I-C_\alpha X_k$ \hfill\COMMENT{residual}
\STATE Compute: $Z_k = X_kR_k$  \hfill\COMMENT{self-preconditioned residual}
\STATE Apply numerical dropping to $Z_k$
\STATE Compute: $\beta_k = \mathrm{tr}(R_k^TC_\alpha Z_k)/\left\Vert
C_\alpha Z_k\right\Vert_F^2$ \hfill\COMMENT{tr$(\cdot)$ denotes the trace}
\STATE Compute: $X_{k+1} = X_k + \beta_k Z_k $
\end{algorithmic}
\end{algorithm}

\subsection{Computation of low-rank approximations} \label{sec:Lan}
For the DDLR-1 method, we  use the Lanczos algorithm~\cite{lanczos} to
compute the low-rank
approximation to $E^TA_0\inv E$ that is of the form $U_k \Lambda_k U_k^T$.
For the DDLR-2 method, the low-rank approximation to
$A_0\inv E$ is of the form $U_kV_k^T$, which can be
computed by applying the Lanczos algorithm on
$E^TA_0^{-2}E$, where $V_k$ is computed and $U_k$ can be obtained by 
$U_k=A_0\inv E V_k$.  Alternatively, for the DDLR-2 method, we can  also use
the 
Lanczos 
bidiagonalization method~\cite[\S 10.4]{GVL-book} to compute $U_k$ and $V_k$ at the same time.
At each step of the Lanczos algorithm, a matrix-vector product is required. This means that for each step, we need to solve linear systems with $A_0$: one solve for the DDLR-1 method and two for the DDLR-2 method.

As is well-known, in the presence of rounding error,
orthogonality in the Lanczos procedure
is quickly lost and a form of reorthogonalization is needed in practice.
In our approach, the partial
reorthogonalization scheme \cite{parlett1979,simon1984} is used.
The cost of this step will not be an issue to the overall performance
when a small number of steps are performed to approximate a few
eigenpairs.
To monitor convergence of the computed eigenvalues,
we adopt the approach used in \cite{fangsaadlanczos}.
Let $\theta_{j}^{(m-1)}$ and $\theta_{j}^{(m)}$ be the Ritz values obtained
in two consecutive Lanczos steps, $m-1$ and $m$.
Assume that we want to approximate $k$ largest
eigenvalues and $k<m$.
Then with a preselected tolerance $\epsilon$, the desired eigenvalues are
considered to have converged if
\begin{equation}
  \left|\frac{\sigma_{m}-\sigma_{m-1}}{\sigma_{m-1}}\right|<\epsilon, \textrm{
  where }\sigma_{m-1}={\displaystyle \sum_{j=1}^{k}\theta_{j}^{(m-1)}} \textrm{
  and } \sigma_{m}={\displaystyle \sum_{j=1}^{k}\theta_{j}^{(m)}}.
  \label{eq:converge}
\end{equation}

\subsection{Parallel implementations: standard mapping} \label{sec:pp}
Considerations of the parallel implementations have been mentioned in the
previous
sections, which suggest several possible schemes for distributing the interface
unknowns. Before discussing these schemes, it will be helpful
to overview the issues at hand. Major computations in building and 
applying the DDLR preconditioners are the following:
\begin{enumerate}
\item[1.] solve with $B_{i,\alpha}$, \hfill (local)
\item[2.] solve with $C_\alpha$, \hfill (nonlocal)
 \item[3.] products with $E^T$ and $E$, \hfill (local)
 \item[4a.] for DDLR-1, applying $G_{k,\theta}\inv$ in \eqref{eq:gkt}, \hfill
(nonlocal)
 \item[4b.] for DDLR-2, products with $U_k$ and $U_k^T$, \hfill (nonlocal)
\item[5.] reorthogonalizations in the Lanczos procedure. \hfill (nonlocal)
\end{enumerate}
The most straightforward mapping we can consider might be to map the unknowns 
of each subdomain to a processor.
If $p$ subdomains are used, global matrices $A$ and $C_\alpha$ or its
approximate inverse $X$ are distributed among the $p$ processors. So, processor 
$i$ will hold $d_i+s_i$ rows of $A$ and $s_i$ rows of $C_\alpha$ or $X$, where 
$d_i$ is the number of the local interior unknowns and $s_i$ is the number of
the local 
 interface unknowns of subdomain~$i$.
In the DDLR-1 method, $U_k\in \RR^{s
\times k}$ is distributed such that processor $i$ will keep $s_i$ rows, while 
in the DDLR-2 method, $d_i+s_i$ rows of $U_k\in \RR^{n \times k}$ will reside in
processor $i$.
For all the nonlocal operations, communication is among all the $p$
processors. The operations labeled by (2.) and (4a.) involve 
interface  to interface  communication, while the operations
(4b.) and (5.) involve communication among all the unknowns. From another 
perspective, the communication
in (4a.), (4b.) and (5.) is of the all-reduction type required by vector inner
products, while the communication in (2.) is point-to-point such as
that in the distributed sparse matrix vector products.
If an iterative process is used for the solve with $C_\alpha$,
it is 
important to select $\alpha $ carefully so as to reach a compromise 
between the number of the inner iterations 
(each of which  requires 
communication) and the number of the outer iterations (each of which involves
solves
with $C_\alpha$). The scalar $\alpha$ will also play a role if an
approximate inverse is used, since the convergence of the MR iterations will be
affected.

\subsection{Unbalanced mapping: interface unknowns together} \label{sec:unbal}
Since communication is required among the interface nodes, an idea that comes 
to mind is to map the interior unknowns of each subdomain
to a processor, and all the interface unknowns to another separated one.
In a case of $p$ subdomains, $p+1$ processors will be used and  $A$ is 
distributed in such a way that processor $i$ owns
the  rows corresponding to the local interior
unknowns for $i=1,\ldots p$, while processor $p+1$ holds the 
rows related to all the interface unknowns.
Thus, $C_\alpha$ or $X$ will reside entirely on the processor 
$p+1$.

A clear advantage of this mapping is that the solve with $C_\alpha$ will 
require no communication.
However, the operations with $E$ and $E^T$ are no longer local.
Indeed, $E^T$ can be viewed as a restriction operator, which ``scatters''
interface data from
processor $p+1$ to the other $p$ processors.
Specifically, referring to \nref{eq:locsys2}, each $y_i$ will be sent to 
processor $i$ from processor $p+1$.
Analogously, the product with $E$, as a prolongation, will perform a dual 
operation that ``gathers'' from processors $1$ to $p$ to processor $p+1$.
In Algorithm~\ref{alg:prec1}, the scatter operation goes before step $2$ and 
the gather operation should be executed after step $4$. Likewise, if we store 
the
vectors in $U_k$ on processor $p+1$, applying $G_{k,\theta}$ will not require
communication but another pair of the ``gather-and-scatter'' operations will be 
needed before and after step 3.
Therefore, at each application of the DDLR-1 preconditioner,
\emph{two}
pairs of the scatter-and-gather operations for the interface
unknowns will be required. 
A middle ground approach is to distribute $U_k$ to processors $1$ to $p$ as it 
is in the standard mapping.
In this way, applying $G_{k,\theta}$ will require communication but only
\emph{one}
pair of the scatter-and-gather operations is necessary.
On the other hand, in the DDLR-2 method, the distribution of $U_k$
should be consistent with that of $A$.

The main issue with this mapping is that
it is hard to achieve load balancing in general. Indeed for
a good balancing, we need to have the interior unknowns of each subdomain and 
all the interface unknowns of
roughly the same size. However, this is difficult to achieve in practice.
The load balancing issue is further complicated by the  fact that 
the equations needed to be solved on processor $p+1$ are completely different
from those on
the other processors.
A remedy to the load balancing issue is to use $q$ processors instead of just 
one dedicated to the 
global interface (a total of $p+q$ processors used in all), which provides a 
compromise. Then, the communication required for solving with $C_\alpha$ and
applying 
$G_{k,\theta}$ is confined within the $q$ processors.

\subsection{Improving  a given preconditioner} \label{sec:improv}
One of the main weaknesses of standard, e.g., ILU-type, preconditioners
is that they are difficult to update. For example, suppose we 
compute a preconditioner  to a given matrix  and find
that it is not accurate enough to yield convergence. In the 
case of ILU we would have essentially to start from the beginning.
For DDLR, improving a given preconditioner is essentially trivial.
For example, the heart of DDLR-1 consists of obtaining a low-rank
approximation the  matrix $G$ defined in \nref{eq:Gmat}. 
Improving this approximation would consist in merely adding a few more
vectors (increasing $k$) and this can be easily achieved in a number of
ways without having to throw away the vectors already computed.

\section{Numerical experiments} \label{sec:exp}
The experiments were conducted on Itasca, an HP ProLiant BL280c G6 Linux
cluster at Minnesota Supercomputing Institute, which has $2,186$ Intel Xeon
X5560 processors. Each processor has four cores, 8 MB
cache, and communicates with memory on a QuickPath Interconnect (QPI) interface.
An implementation of the DDLR preconditioners was
written in C/C++ with the Intel Math Kernel Library, the Intel MPI library
and PETSc \cite{petsc-user-ref,petsc-web-page,petsc-efficient}, compiled by the 
Intel MPI compiler using the -O3 optimization level.

The accelerators used were the conjugate gradient (CG)
method when both the matrix and the preconditioner are SPD, and 
the generalized minimal residual (GMRES) method with a restart dimension of $40$, denoted by
GMRES$(40)$, for the indefinite cases.
Three types of preconditioners were compared in our experiments: 1) the 
DDLR
preconditioners, 2) the pARMS method \cite{NLAA:pARMS},
and 3) the RAS preconditioner \cite{RAS} (with overlapping).
Recall that for an SPD matrix, the DDLR preconditioners  given by
\eqref{eq:Mprec} and \eqref{eq:Form2a} will also be SPD if the assumptions 
in
Propositions \ref{prop:spd-1} and \ref{prop:spd-2} are satisfied.
However, these propositions will not
hold when the solves with $A_0$ are approximate, which is
typical in practice.
Instead, the positive definiteness can be determined by checking if the
largest eigenvalue is less than one for DDLR-1
or by checking the positive definiteness of 
$H_k$ for DDLR-2.
DDLR-1 was always used with $\theta=\lambda_{k+1}$.

Each $B_{i,\alpha}$ was reordered by the
approximate minimum degree ordering (AMD) \cite{AmestoyDavisDuff96,
AmestoyDavisDuff04,Davis:2006:DMS:1196434} to reduce fill-ins 
and then we simply used an incomplete Cholesky or LDL factorization as the 
local solver. 
A more efficient and robust 
local
solver, for example, the ARMS approach in \cite{Saad-Suchomel-ARMS}, can lead to
better performance in terms of both the memory requirement and the speed. 
However, this has not been implemented in our current code.
A typical setting of the scalar $\alpha$ for $C_{\alpha}$ and $B_{i,\alpha}$  is
$\alpha=1$, which in general gives the best overall
performance, the exceptions being the three cases shown in Section
\ref{sec:gen}, for which choosing $\alpha > 1$ improved the 
convergence.
Regarding the solves with $C_{\alpha}$, using the approximate inverse is 
generally more efficient than the Chebyshev iterations, especially in the 
iteration phase.
However, computing the approximate inverse can be costly, in particular for the 
indefinite 3-D cases. The standard mapping was
adopted unless specially stated, which in general gave 
better performance than the unbalanced mapping. The behavior of these two 
types of mappings will be
analyzed by the results in Table~\ref{tab:mapping}.
In the Lanczos algorithm, the convergence was checked every $10$
iterations and the  tolerance $\epsilon$ in \eqref{eq:converge} used for the 
checking 
was $10^{-4}$. In addition, the maximum 
number of the Lanczos steps was five times the number of the requested 
eigenvalues.

For pARMS, the ARMS method was used to be the local preconditioner and the 
Schur 
complement method was used as the global preconditioner, where the reduced 
system was solved by a few inner Krylov subspace iterations preconditioned by 
the block-Jacobi preconditioner. For the details of these options in pARMS, we 
refer the readers to \cite{parms2002,Saad-Suchomel-ARMS}.
We point out  that when the inner iterations are enabled,
flexible Krylov subspace methods will be required for the outer iterations, 
since the preconditioning
is no longer fixed from one outer iteration to the next. So, the flexible 
GMRES~\cite{fgmres} was used.
For the RAS method, ILU($k$) was used as the local solver, and a one-level 
overlapping between subdomains was used.
Note that the RAS preconditioner is nonsymmetric even for a symmetric matrix, 
so that GMRES was used with it.

We first report on the results of solving the linear
systems
from a 2-D and a 3-D  PDEs on regular meshes.
Next, we will show the results for solving a sequence of
general sparse symmetric linear systems.
For all the problems, a parallel multilevel $k$-way graph partitioning algorithm
from Par\textsc{Metis} \cite{METIS-SIAM,Karypis199871} was used for the DD.
Iterations were stopped whenever the residual norm had been
reduced by $6$ orders of magnitude or the maximum number of iterations allowed,
which is $500$, was exceeded.
The results are shown in Tables \ref{tab:spd2d3d}, 
\ref{tab:indef2d3d} 
and
\ref{tab:general}, where
all timings are reported in seconds and `\texttt{F}' indicates non-convergence 
within the maximum allowed
number of steps.
When comparing the preconditioners, the following factors are considered:
1) fill-ratio, i.e., the ratio of the number of nonzeros required to store
    a preconditioner to the number of nonzeros in the original matrix,
2) time for building preconditioners,
3) the number of iterations and
4) time for the iterations.

\subsection{Model problems}
We examine a 2-D and a 3-D PDE,
\begin{align} \label{eq:2d3d-pde}
-\Delta u-cu &= f\:\textrm{ in }\Omega, \notag \\
u &= 0 \textrm{ on } \partial \Omega,
\end{align}
where $\Omega = \left(0,1\right)^2$ or $\Omega =
\left(0,1\right)^3$, 
and $\partial \Omega$ is the boundary.
We take the $5$-point (or $7$-point) centered difference approximation.
To begin with, we solve \eqref{eq:2d3d-pde} with $c=0$. The matrix is SPD, so 
that we use DDLR with CG.
Numerical experiments were carried out to compare the performance of DDLR
with those of pARMS and RAS. The results are shown in
Table \ref{tab:spd2d3d}.
The mesh sizes, the number of processors (Np), the rank (rk),
the fill-ratios (nz), the numbers of iterations
(its), the time for building the preconditioners (p-t) and the time
for iterations (i-t)  are  tabulated.
We tested the problems on 6 2-D and 6 3-D meshes of increasing sizes, 
where the number of processors was growing proportionally such that the problem 
size on each processor was kept roughly the same. 
This can serve as a \emph{weak scaling} test.
We increased the rank $k$ used in DDLR with the meshes sizes.
The fill-ratios of DDLR-1 and pARMS were controlled to be 
roughly equal, whereas the fill of DDLR-2 was much higher, which 
comes mostly from the matrix $U_k$ when $k$ is large.
For pARMS, the inner Krylov subspace dimension used was $3$.

The time for building DDLR is much higher and it grows with the rank and 
the number of the processors. In contrast, the time to build pARMS and 
RAS is roughly constant.
This set-up time for DDLR  is typically dominated by
the Lanczos algorithm, where solves with $B_{i,\alpha}$ and $C_\alpha$ are
required at each iteration. Moreover, when $k$
is large, the cost of reorthogonalization becomes significant.
As shown in Table \ref{tab:spd2d3d},  DDLR-1  and  pARMS
 were more robust as they succeeded for all the 2-D and 3-D cases, while
 DDLR-2  failed for the largest 2-D case and  RAS 
failed for the three largest ones.
For most of the 2-D problems, DDLR-1/CG  achieved convergence in the
fewest iterations and the best iteration time.
For the 3-D problems, DDLR-1 required more iterations but a 
performance gain was still achieved  in terms of the reduced 
iteration time. Exceptions were the two largest 3-D problems, where RAS/GMRES 
yielded the best iteration time.

\begin{table}[h]
\caption{Comparison between {\rm DDLR}, {\rm pARMS} and {\rm RAS} preconditioners for
solving SPD linear systems from the 2-D/3-D PDE with the {\rm 
CG} or {\rm GMRES(40)} method. \label{tab:spd2d3d}}
\begin{center}
\begin{tabular}{r|r|ccrcc|ccrcc}
\hline
\multirow{2}{*}{Mesh}  & \multirow{2}{*}{Np}  &
\multicolumn{5}{c|}{DDLR-1} &
\multicolumn{5}{c}{DDLR-2}\tabularnewline
 &  & rk  & nz  & its  & p-t  & i-t  & rk  & nz  & its  & p-t  &
i-t\tabularnewline
\hline
$128^{2}$  & 2  & 8  & 6.6  & 15  & .209  & .027  & 8  & 8.2  & 30  & .213  &
.031\tabularnewline
$256^{2}$  & 8  & 16  & 6.6  & 34  & .325  & .064  & 16  & 9.7  & 69  & .330  &
.083\tabularnewline
$512^{2}$  & 32  & 32  & 6.8  & 61  & .567  & .122  & 32  & 13.0  & 132  & .540
& .194 \tabularnewline
$1024^{2}$  & 128  & 64  & 7.0 & 103 & 1.12 & .218 & 64  & 19.3 & 269 & 1.03 &
.570 \tabularnewline
$1448^{2}$  & 256  & 91  & 7.2 & 120 & 1.67 & .269 & 91  & 24.7 & 385 & 1.72 &
1.05 \tabularnewline
$2048^{2}$  & 512  & 128  & 7.6 & 168 & 3.02 & .410 & 128  & 32.2 & \tF & - &
- \tabularnewline
\hline
$25^{3}$  & 2  & 8  & 7.2  & 11  & .309  & .025  & 8  & 8.3  & 17  & .355  &
.021 \tabularnewline
$50^{3}$  & 16  & 16  & 7.5  & 27  & .939  & .064  & 16  & 9.3 & 52  & .958  &
.076\tabularnewline
$64^{3}$  & 32  & 16  & 7.4  & 36  & 1.06  & .089  & 16  & 9.2 & 67  & 1.07  &
.102 \tabularnewline
$100^{3}$  & 128  & 32  & 8.0  & 52 & 1.57 & .136  & 32  & 11.5  & 101 & 1.48  &
.190\tabularnewline
$126^{3}$  & 256  & 32  & 8.2  & 65 & 2.07 & .178  & 32  & 12.5  & 126 & 1.87
& .265 \tabularnewline
$159^{3}$  & 512  & 51  & 8.7  & 85 & 2.92 & .251  & 51  & 14.2  & 156 & 2.50
& .387 \tabularnewline
\hline
\end{tabular}

\vspace{1em}

\begin{tabular}{r|r|ccrcc|ccrcc}
\hline
\multirow{2}{*}{Mesh}  & \multirow{2}{*}{Np}  & \multicolumn{5}{c|}{pARMS} &
\multicolumn{5}{c}{RAS}\tabularnewline
 &  & \phantom{rk}  & nz  & its  & p-t  & i-t  & \phantom{rk}  & nz  & its  &
p-t  & i-t\tabularnewline
\hline
$128^{2}$  & 2  & \phantom{8}  & 6.7  & 15  & .062  & .037  & \phantom{8}  &
\phantom{1}2.7 & 40 & .003 & .032 \tabularnewline
$256^{2}$  & 8  & \phantom{16}  & 6.7  & 30  & .066 & .082 & \phantom{16}  &
\phantom{1}2.7
 & 102 & .004 & .072 \tabularnewline
$512^{2}$  & 32  & \phantom{32}  & 6.9  & 52  & .072 & .194 & \phantom{32} &
\phantom{1}2.7
 & 212 & .005 & .157 \tabularnewline
$1024^{2}$  & 128  & \phantom{64}  & 6.6 & 104 & .100 & .359 & \phantom{64} &
\phantom{1}2.7
 & \tF & .008 & - \tabularnewline
$1448^{2}$  & 256  & \phantom{64}  & 6.6 & 247 & .073 & .820 & \phantom{64} &
\phantom{1}2.7
 & \tF & .011 & - \tabularnewline
$2048^{2}$  & 512  & \phantom{128}  & 6.8 & 282 & .080 & 1.06 & \phantom{128}
& \phantom{1}2.7 & \tF & .015 & - \tabularnewline
\hline
$25^{3}$  & 2  & \phantom{8}  & 7.3  & 9  & .100  & .032  & \phantom{8} &
\phantom{1}5.9
& 13 & .004 & .041 \tabularnewline
$50^{3}$ & 16  & \phantom{16}  & 8.1  & 17 & .179 & .095 & \phantom{16} &
\phantom{1}6.7
& 28 & .006 & .071 \tabularnewline
$64^{3}$  & 32  & \phantom{16}  &  8.2 & 20 & .142 & .121 & \phantom{16} &
\phantom{1}6.7
& 34 & .007 & .103 \tabularnewline
$100^{3}$  & 128  & \phantom{32}  & 8.3 & 29 & .170 & .198 & \phantom{32} &
\phantom{1}6.7
& 51 & .011 & .148 \tabularnewline
$126^{3}$  & 256  & \phantom{32}  & 8.4 & 34 & .166 & .216 & \phantom{32} &
\phantom{1}6.7
& 60 & .014 & .127 \tabularnewline
$159^{3}$  & 512  & \phantom{111}  & 8.5 & 40 & .179 & .275 & \phantom{111} &
\phantom{1}6.7 & 83 & .019 & .183 \tabularnewline
\hline
\end{tabular}
\end{center}
\end{table}

Next, we consider solving symmetric indefinite problems by setting $c>0$ in
\eqref{eq:2d3d-pde},
which corresponds to shifting the discretized negative Laplacian
 by subtracting  $\sigma I$ with a certain
$\sigma>0$.
In this set of experiments, we reduce the size of the shift as the
problem size increases in order to make the problems fairly difficult but not
too difficult to solve for all the methods.
We used higher ranks in the two DDLR methods and a higher inner iteration 
number, which was $6$, in pARMS. 
Results are reported in Table~\ref{tab:indef2d3d}.
From there  we can  see that DDLR-2 did
not perform well as it failed for almost all
the problems. Second, RAS failed for all the 2-D cases and
three 3-D cases. But for the  three cases where it worked, it 
yielded the best iteration time. Third,  DDLR-1  achieved convergence
in all the cases whereas  pARMS failed for two 2-D cases. Comparison
between DDLR-1 and pARMS shows a similar result as in the previous set of 
experiments:
for the 2-D cases, DDLR-1 required fewer iteration and less
iteration time,
while for the 3-D cases, it might require more iterations but still less
iteration time.

\begin{table}[h]
\caption{Comparison between {\rm DDLR}, {\rm pARMS} and
{\rm RAS} preconditioners for
solving symmetric indefinite linear systems from the 2-D/3-D  PDEs with 
the {\rm GMRES(40)} method.
\label{tab:indef2d3d}}
\begin{center}
 \small
\begin{tabular}{r|r|c|ccrcc|ccrcc}
\hline
\multirow{2}{*}{Mesh}  & \multirow{2}{*}{Np}  & \multirow{2}{*}{$\sigma$} &
\multicolumn{5}{c|}{DDLR-1} & \multicolumn{5}{c}{DDLR-2}\tabularnewline
 &  &  & rk  & nz  & its  & p-t  & i-t  & rk  & nz  & its  & p-t  &
i-t\tabularnewline
\hline
$128^{2}$  & 2  & 1e-1  & 16  & 6.8  & 18  & .233  & .034  & 16  & 13.2  &
146 & .310  & .234 \tabularnewline
$256^{2}$  & 8  & 1e-2 & 32  & 6.8  & 38  & .674  & .080  & 16  & 13.0  & \tF
& 1.01 & -\tabularnewline
$512^{2}$  & 32  & 1e-3 & 64  & 7.1  & 48  & 1.58  & .105  & 64  & 19.4  &
\tF & 1.32  & - \tabularnewline
$1024^{2}$  & 128 & 2e-4 & 128  & 7.6 & 68 & 4.15 & .160 & 128  & 32.3 & \tF
& 4.45 & - \tabularnewline
$1448^{2}$  & 256  & 5e-5 & 182  & 8.1 & 100 & 7.14 & .253 & 182  & 43.2 &
\tF & 7.77 & - \tabularnewline
$2048^{2}$  & 512  & 2e-5 & 256  & 8.8 & 274 & 12.6 & .749 & 256  & 58.4 &
\tF & 13.1 & - \tabularnewline
\hline
$25^{3}$  & 2  & .25 & 16  & 8.3  & 29  & .496  & .099  & 16  & 9.6  & 62  &
.595  & .130 \tabularnewline
$50^{3}$  & 16  & 7e-2 & 32  & 8.2  & 392  & 1.38  & 1.19  & 32  & 10.2  &
\tF & 1.66 & -\tabularnewline
$64^{3}$  & 32  & 3e-2 & 64  & 8.9 &  201 &  2.26 & .688  & 64  & 16.3  & \tF
& 2.08 & - \tabularnewline
$100^{3}$  & 128  & 2e-2 & 128 & 11.4 & 279 & 5.17 & 1.08 & 128 & 28.7 & \tF &
5.29 & - \tabularnewline

$126^{3}$  & 256  & 7e-3 & 128 & 12.8 & 255 & 5.85 & 1.10 & 128 & 28.3 & \tF &
6.01 & - \tabularnewline

$159^{3}$  & 512  & 5e-3 & 160 & 13.5 & 387 & 8.60 & 1.71 & 160 & 33.0 & \tF &
8.33 & - \tabularnewline
\hline
\end{tabular}

\vspace{1em}

\begin{tabular}{r|r|c|ccrcc|ccrcc}
\hline
\multirow{2}{*}{Mesh}  & \multirow{2}{*}{Np}  & \multirow{2}{*}{$\sigma$}  &
\multicolumn{5}{c|}{pARMS} & \multicolumn{5}{c}{RAS}\tabularnewline
 &  &  & \phantom{rk}  & nz  & its  & p-t  & i-t  & \phantom{rk}  & nz  & its  &
p-t  & i-t\tabularnewline
\hline
$128^{2}$  & 2  & 1e-1  & \phantom{32}  & 11.4  & 76  & .114  & .328  &
\phantom{8}  & \phantom{1}2.7 & \tF & .003 & - \tabularnewline
$256^{2}$  & 8  & 1e-2 & \phantom{16}  & 13.9  & \tF  & - & -
& \phantom{16}  & \phantom{1}2.7 & \tF & .004 & - \tabularnewline
$512^{2}$  & 32  & 1e-3 & \phantom{32}  & 12.3  & 298  & .181 & 1.53 &
\phantom{32} & \phantom{1}2.7 & \tF & .005 & - \tabularnewline
$1024^{2}$  & 128  & 2e-4 & \phantom{64}  & 12.5 & 232 & .230 & 1.46 &
\phantom{64} & \phantom{1}2.7 & \tF & .008 & - \tabularnewline
$1448^{2}$  & 256 & 5e-5 & \phantom{64}  & 12.5 & \tF & - & - &
\phantom{64} & \phantom{1}2.7 & \tF & .011 & - \tabularnewline
$2048^{2}$  & 512 & 2e-5 & \phantom{64 }  & 12.6 & 314 & .195 & 2.13 &
\phantom{64 } & \phantom{1}2.7 & \tF & .015 & - \tabularnewline
\hline
$25^{3}$  & 2  & .25  & \phantom{16}  & 8.3  & 100  & .156  & .599  &
\phantom{16}  & \phantom{1}5.9 & 108 & .004 & .123 \tabularnewline
$50^{3}$  & 16  & 7e-2 & \phantom{16}  & 8.9  & 448  & .142 & 2.59 &
\phantom{16} & \phantom{1}6.7 & \tF & .006 & - \tabularnewline
$64^{3}$  & 32 & 3e-2 & \phantom{16}  &  8.9 & 130  & .115 & .784 &
\phantom{16} & \phantom{1}6.7 & 252 & .007 & .375 \tabularnewline

$100^{3}$  & 128 & 2e-2 & \phantom{128}  & 9.4 & 187 & .137 & 1.24 &
\phantom{32} & \phantom{1}6.7 & 343 & .011 & .541 \tabularnewline

$126^{3}$  & 256  & 7e-3 & \phantom{128}  & 10.6 & 340 & .137 & 2.74 &
\phantom{32} & \phantom{1}6.7 & \tF & .014 & - \tabularnewline

$159^{3}$  & 512  & 5e-3 & \phantom{128}  & 10.8 & 329 & .148 & 2.85 &
\phantom{312} & \phantom{1}6.7 & \tF & .019 & - \tabularnewline
\hline
\end{tabular}
\end{center}
\end{table}

In all the previous tests, DDLR-1  was used with the
standard mapping. In the next set of experiments, we  examined the behavior
of the unbalanced mapping discussed in Section
\ref{sec:unbal}. In these experiments, we tested the problem on a  $128 \times 
128$ mesh
and a
 $25 \times 25 \times 25$ mesh. Both of them were divided into $128$ subdomains.
Note here that the problem size per processor is remarkably small. This was made
on purpose since it can make the communication cost more significant (and
likely to be dominant) in the overall cost for the solve with $C_\alpha$ such 
that it can make the effect of the unbalanced mapping more prominent.
Table \ref{tab:mapping} lists the iteration time for solving the SPD problem 
using the standard mapping and the unbalanced mapping with different settings.
Two solution methods for $C_{\alpha}$ were tested, the
one with the
approximate inverse and the preconditioned Chebyshev iterations (5 iterations
were used per solve). In the
unbalanced mapping, $q$ processors were used dedicated to the interface
unknowns ($p+q$ processors were used totally). 
The unbalanced mapping was
tested with 8 different $q$ values from $1$ to $96$. $q=1$ is a 
special case
where no communication is involved in the solve with $C_\alpha$. The matrix $U_k$ 
was stored on the $p$ processors, so that only
one pair of the scatter-and-gather communication was required at each outer 
iteration as
discussed in Section \ref{sec:unbal}.
The standard mapping is indicated by $q=0$.

\begin{table}[ht]
\caption{Comparison of the iteration time (in
milliseconds) between the standard mapping and the
unbalanced
mapping for solving 2-D/3-D SPD PDE problems by the DDLR-1-CG method. \label{tab:mapping}}
\begin{center}
\begin{tabular}{c|c|ccccccccc}
\hline
Mesh & $C_{\alpha}^{-1}$ & $q=0$ & 1 & 2 & 4 & 8 & 16 & 32 & 64 &
96\tabularnewline
\hline
\multirow{2}{*}{$128^{2}$} & \texttt{AINV} & 9.0 & 53.4 & 27.8 & 15.4  &
13.9 & 10.8  & \textbf{9.1}  & 9.4  & 13.5\tabularnewline
 & \texttt{Cheb} & 9.2 & 116.8  & 60.7  & 29.9  & 15.8  & 10.7
& 10.0  & \textbf{8.3}  & 9.1 \tabularnewline
\hline
\multirow{2}{*}{$25^{3}$} & \texttt{AINV} & 15.1 & 119.7 & 66.3 & 34.7
& 26.0 & 19.2 & 17.2 & \textbf{14.8} & 18.2 \tabularnewline
 & \texttt{Cheb} & 13.8 & 368.3 & 166.0 & 78.3 & 38.5 & 20.4 & 14.4 &
\textbf{12.3} & 13.5\tabularnewline
\hline
\end{tabular}
\end{center}
\end{table}

As the results indicated, the iteration time kept decreasing at
beginning as $q$ increased but after some
point it started to increase. This is a typical situation corresponding to
the balance between communication and computation: when $q$ is small, the 
amount of
computation on each of the $q$ processors is high and it dominates the overall
cost, so that the overall cost will keep being reduced as $q$ increases until 
the point when
the communication cost starts to affect the overall performance. The optimal
numbers of the interface processors that yielded the best iteration time are 
shown in bold in Table \ref{tab:mapping}. For these two cases, the optimal 
iteration time with the unbalanced mapping was
slightly better than that with the standard mapping. However,
we need to point out that this is not a typical case in practice. For all the 
other tests in this section, we used the standard mapping in the 
DDLR-1 preconditioner.

\subsection{General matrices} \label{sec:gen}

We selected $12$ symmetric matrices from the University of Florida sparse matrix
collection \cite{Davis:UFM} for the following tests. Table 
\ref{tab:matrices}
lists the name, the order (N), the number of nonzeros (NNZ), and a short 
description for
each matrix. If the actual right-hand side is not provided, an 
artificial one was created as $b=Ae$, where $e$ is a random 
vector.
 
\begin{table}[h]
\caption{Names, orders (N), numbers of nonzeros (NNZ)
and short descriptions of the test matrices.}
\begin{center}
\begin{tabular}{l|r|r|r}
\hline
\multicolumn{1}{c|}{MATRIX} & \multicolumn{1}{c|}{N} & \multicolumn{1}{c|}{NNZ}
& \multicolumn{1}{c}{DESCRIPTION}\tabularnewline
\hline
Andrews/Andrews  & 60,000    & 760,154   &  computer graphics problem
\tabularnewline
UTEP/Dubcova2    & 65,025    & 1,030,225 &  2-D/3-D PDE problem
\tabularnewline
Rothberg/cfd1    & 70,656    & 1,825,580 &  CFD problem \tabularnewline
Schmid/thermal1  & 82,654    & 574,458   &  thermal problem \tabularnewline
Rothberg/cfd2    & 123,440   & 3,085,406 &  CFD problem \tabularnewline
UTEP/Dubcova3    & 146,689   & 3,636,643 &  2-D/3-D PDE problem \tabularnewline
Botonakis/thermo\_TK & 204,316  & 1,423,116 &  thermal problem
\tabularnewline
Wissgott/para\_fem & 525,825 & 3,674,625 & CFD problem \tabularnewline
CEMW/tmt\_sym & 726,713 & 5,080,961 &  electromagnetics problem
\tabularnewline
McRae/ecology2 & 999,999 & 4,995,991 &  landscape ecology problem
\tabularnewline
McRae/ecology1 & 1,000,000 & 4,996,000 &  landscape ecology problem
\tabularnewline
Schmid/thermal2  & 1,228,045 & 8,580,313 &  thermal problem \tabularnewline
\hline
\end{tabular}
\end{center}
\label{tab:matrices}
\end{table}

Table~\ref{tab:general} shows the results  for each 
problem.
DDLR-1 and DDLR-2  were used with GMRES($40$)
for three problems {\tt tmt\_sym}, {\tt ecology1} and {\tt ecology2},
where the preconditioners were found not to be SPD, while for
the other problems CG  was applied.
We set the scalar $\alpha=2$  for
two problems {\tt ecology1} and {\tt ecology2}, where it turned out to reduce
the numbers of iterations, but for elsewhere we use $\alpha=1$.
As shown by the results, DDLR-1 achieved convergence
for all the cases, whereas the other three preconditioners all had 
failures for a few cases.
Similar to the experimental results for the model problems, the DDLR 
preconditioners required more time to construct. Compared with  
pARMS and RAS, DDLR-1 achieved
time savings in the iteration phase for 7 (out of 12) problems and 
DDLR-2 did so for 4 cases.

\begin{table}[h]
\caption{Comparison among {\rm DDLR}, {\rm pARMS} and {\rm RAS}
  preconditioners for solving general sparse symmetric linear systems
  along with  {\rm CG} or  {\rm GMRES$(40)$}. \label{tab:general}}
\begin{center}
\begin{tabular}{l|r|ccrcc|ccrcc}
\hline
\multirow{2}{*}{Matrix}  & \multirow{2}{*}{Np}  & \multicolumn{5}{c|}{DDLR-1} &
\multicolumn{5}{c}{DDLR-2}\tabularnewline
 &  & rk  & nz  & its  & p-t  & i-t  & rk  & nz  & its  & p-t  &
i-t\tabularnewline
\hline
Andrews  & 8  & 8  & 4.7  & 33  & .587  & .220  & 8  & 5.2  & 53  & .824 &
.175\tabularnewline
Dubcova2  & 8  & 16  & 3.5  & 18  & .850  & .054  & 16 & 4.5 & 44  & .856
& .079 \tabularnewline
cfd1  & 8  & 8  & 18.1 & 17 & 7.14 & .446 & 8 & 18.4 & 217 & 6.44 & 2.97
\tabularnewline
thermal1  & 8 & 16 & 6.0  & 48  & .493  & .145 &  16 &  8.3 & 126 & .503 &
.234 \tabularnewline
cfd2  & 16  & 8 & 13.2 & 12 & 4.93  & .232  &  8 & 13.4 & \tF  &  5.11 & -
\tabularnewline
Dubcova3 & 16  & 16 & 2.6 & 16 & 1.70 & .061 & 16 & 3.2 & 44 & 1.71 & .107
\tabularnewline
thermo\_TK & 16  & 32  & 6.4 & 24 & .568 & .050 & 32  & 10.8 & 63
& .537 & .096
\tabularnewline
para\_fem  & 16  & 32 & 7.8 & 59 & 4.02 & .777 & 32 & 12.3 & 159
& 4.12 & 1.35
\tabularnewline
tmt\_sym  & 16  & 16 & 7.3 & 33 & 5.56 & .668 & 16 & 9.5 & 62 & 5.69 &
.790 \tabularnewline
ecology2  & 32  & 32 & 8.9 & 39 & 3.67  & .433 & 32 & 15.2 & 89
& 3.79 & .709
\tabularnewline
ecology1  & 32  & 32  & 8.8 & 40 & 3.48 & .423 & 32  & 15.1
& 82 & 3.59 & .656
\tabularnewline
thermal2 & 32 & 32 & 6.8 & 140 & 5.06 & 2.02 & 32 & 11.3 & \tF & 5.11 & -
\tabularnewline
\hline
\end{tabular}

\vspace{1em}

\begin{tabular}{r|r|ccrcc|ccrcc}
\hline
\multirow{2}{*}{Matrix}  & \multirow{2}{*}{Np}  & \multicolumn{5}{c|}{pARMS} &
\multicolumn{5}{c}{RAS}\tabularnewline
 &  & \phantom{rk}  & nz  & its  & p-t  & i-t  & \phantom{rk}  & nz  & its  &
p-t  & i-t\tabularnewline
\hline
Andrews  & 8  & \phantom{8}  & 4.3  & 15  & .217  & .109  & \phantom{8}  & 3.6
& 19 & .010 & .073 \tabularnewline
Dubcova2  & 8  & \phantom{32}  & 3.5  & 25  & .083 & .090 & \phantom{32}  & 3.5
& 43 & .008 & 0.11 \tabularnewline
cfd1  & 8 & \phantom{64}  & 16.1 & \tF & .091 & - & \phantom{64} & 10.6
& 153 & .013 & 3.55 \tabularnewline
thermal1  & 8 & \phantom{8}  & 5.4  & 39  & .089  & .153  & \phantom{8}  & 4.6
& 156 & .006 & .235 \tabularnewline
cfd2  & 16  & \phantom{16}  & 26.0  & \tF & .120 & - & \phantom{16}  & 11.9
& 310 & .012 & 3.26 \tabularnewline
Dubcova3 & 16  & \phantom{32}  & 2.6 & 37 & .130 & .200 & \phantom{32}  & 4.2
& 39 & .013 & .212 \tabularnewline
thermo\_TK & 16  & \phantom{32}  & 4.9 & 16 & .048 & .035 & \phantom{32} & 5.5
& 34 & .004 & .067 \tabularnewline
para\_fem  & 16  & \phantom{32}  & 6.5 & 89 & .586 & 1.36 & \phantom{32} & 5.1
& 247 & .019 & 1.18 \tabularnewline
tmt\_sym  & 16  & \phantom{32}  & 6.9 & 16 & .587 & .361 & \phantom{32}  & 3.7
& 26 & .026 & .222 \tabularnewline
ecology2  & 32  & \phantom{32}  & 9.9 & 15 & .662 & .230 & \phantom{32}  & 5.8
& 28 & .017 & .165 \tabularnewline
ecology1  & 32  & \phantom{32}  & 10.0 & 14 & .664 & .220 & \phantom{32}  & 5.8
& 27 & .017 & .161 \tabularnewline
thermal2  & 32  & \phantom{32}  & 6.1 & 205 & .547 & 3.70 & \phantom{32}  & 4.7
& \tF & .025 & -\tabularnewline
\hline
\end{tabular}
\end{center}

\end{table}

\section{Conclusion} \label{sec:concl}

This paper  presented a  preconditioning method for  solving distributed
symmetric sparse  linear  systems, based on an approximate inverse of the
original matrix which exploits the domain decomposition method and low-rank
approximations.  Two low-rank approximation strategies are discussed,
called  DDLR-1 and DDLR-2.
In terms of the number of iterations and iteration  time,
experimental results  indicate that  for SPD
systems,  the DDLR-1  preconditioner  can  be   an efficient
 alternative to other domain decomposition-type approaches such as one
based on  distributed Schur complements
(as in  pARMS), or on the RAS preconditioner.
Moreover,  this  preconditioner appears  to be  more
robust  than  the pARMS method and the RAS method for  indefinite
problems.  

The DDLR preconditioners require more time
to build than other DD-based preconditioners.
However, one must take a number of  other factors into account.
First, some improvements can be made to reduce the set-up time. 
For example, more efficient local solvers such as ARMS  can be used instead of the current ILUs; 
vector processing such as GPU computing can accelerate the computations of the low-rank 
corrections; and more efficient algorithms than the Lanczos method, e.g., randomized techniques \cite{halko-al-survey-11},
 can be exploited for computing the eigenpairs. 
Second, there are many applications
in which many systems with the same matrix must be solved. In this case
more expensive but more effective preconditioners may be
justified because their cost can be amortized. 
Finally, another important factor  touched upon briefly in 
Section~\ref{sec:improv} is that
the preconditioners discussed here are more easily
updatable than traditional ILU-type or DD-type preconditioners. 

\section*{Acknowledgements}
The authors  are grateful for resources from the University of Minnesota
Supercomputing Institute and for assistance with the computations.
The authors would like to thank the PETSc team for their help with the
implementation.

\bibliographystyle{siam}


\bibliography{local}
\end{document}